\documentclass[11pt,a4paper,amsthm]{amsart}
\usepackage{amsthm}
\usepackage{amssymb}
\usepackage{amsmath}
\usepackage{tikz}
\usepackage{cite}
\usepackage{url}
\usepackage{a4wide}

\DeclareMathOperator{\ro}{r}
\DeclareMathOperator{\EE}{EE}
\DeclareMathOperator{\E}{E}
\DeclareMathOperator{\ssum}{sum}
\DeclareMathOperator{\M}{M}
\DeclareMathOperator{\R}{\mathbb{R}}
\DeclareMathOperator{\tr}{tr}

\allowdisplaybreaks

\newtheorem{conj}{Conjecture}
\newtheorem{thm}{Theorem}
\newtheorem{lem}[thm]{Lemma}
\newtheorem{cor}[thm]{Corollary}
\theoremstyle{remark}
\newtheorem{rem}{Remark}
\newtheorem{defn}{Definition}

\begin{document}

\thispagestyle{plain}

\title{Spectral moments of trees with given degree sequence}

\author{Eric Ould Dadah Andriantiana}
\address{Eric Ould Dadah Andriantiana\\
Department of Mathematical Sciences\\
Stellenbosch University\\
Private Bag X1 \\
Matieland 7602\\
South Africa
}
\email{ericoda@sun.ac.za}
\thanks{Supported by the German Academic Exchange
Service (DAAD), in association with the African Institute for Mathematical Sciences
(AIMS). Code No.: A/11/07858.}

\author{Stephan Wagner}
\address{Stephan Wagner \\
Department of Mathematical Sciences \\
Stellenbosch University \\
Private Bag X1 \\
Matieland 7602 \\
South Africa
}
\email{swagner@sun.ac.za}
\thanks{Supported by the National Research Foundation of South Africa under grant number 70560.}

\date{\today}

\begin{abstract}
Let $\lambda_1,\dots,\lambda_n$ be the eigenvalues of a graph $G$. For any 
$k\geq 0$, the $k$-th spectral moment of $G$ is defined by $\M_k(G)=\lambda_1^k+\dots+\lambda_n^k$. We use the fact that $\M_k(G)$ is also the number of closed walks of length $k$ in $G$ to show that among trees $T$ whose degree sequence is $D$ or majorized by $D$, $\M_k(T)$ is maximized by the greedy tree with degree sequence $D$ (constructed by assigning the highest degree in $D$ to the root, the second-, third-, \dots highest degrees to the neighbors of the root, and so on) for any $k\geq 0$. Several corollaries follow, in particular a conjecture of Ili\'c and Stevanovi\'c on trees with given maximum degree, which in turn implies a conjecture of Gutman, Furtula, Markovi\'c and Gli\v{s}i\'c  on the Estrada index of such trees, which is defined as $\EE(G)=e^{\lambda_1}+\dots+e^{\lambda_n}$.
\end{abstract}

\maketitle

\section{Introduction}
\label{Sec:i}
Let $G$ be a graph with adjacency matrix $A$, and let $\lambda_1,\dots,\lambda_n$ be the 
eigenvalues of $A$. The $k$-th spectral moment of $G$ is defined as
\begin{equation}
\label{Eq:EMk}
\M_k(G)=\sum_{k=0}^n\lambda_i^k.
\end{equation}
A \emph{walk} of length $k$ in a graph $G$ is any sequence $w_1w_2\dots w_{k+1}$ of vertices of $G$ such that $w_iw_{i+1}$ is an edge in $G$ for $i=1,\dots,k$. Since $\tr(A^k)=\M_k(G)$, where $\tr(A^k)$ is the trace of the $k$-th power of $A$, 
$\M_k(G)$ is (see \cite{Book_Spect}) exactly the number of closed walks (walks that start and end at the same vertex) of length $k$ in $G$.
The spectral moments of $G$ are also closely related to the so-called \emph{Estrada index} \cite{EstEst}, which is defined as
\begin{equation}
\label{Eq:EE}
 \EE(G)=\sum_{i=1}^ne^{\lambda_i}.
\end{equation}
It follows from \eqref{Eq:EMk}, \eqref{Eq:EE} and the power series expansion of the exponential function that
\begin{equation}
 \EE(G)=\sum_{i=1}^n\sum_{k=0}^{\infty}\frac{\lambda_i^k}{k!}=\sum_{k=0}^{\infty}\frac{\M_k(G)}{k!}.
\end{equation}
Ernesto Estrada \cite{EstradaIntro} introduced the parameter $\EE$ in 2000  and showed how it can be used to study aspects of molecular structures such as the degree of folding of proteins, see also \cite{EstradaIntro2,EstradaInt4}. Applications of $\EE$ expanded quickly to the study of  complex networks \cite{EsCompl} and quantum chemistry \cite{QUA:QUA20850}. See \cite{Gutman2011Estrada} for a recent survey on the Estrada index.

Let us also define a generalization of the graph invariant $\EE$: for any function $f: \R \to \R$, we set
\begin{equation}\label{Eq:efdefi}
 \E_f(G)=\sum_{i=1}^nf(\lambda_i).
\end{equation}
Obviously, we obtain the $k$-th spectral moment for $f(x) = x^k$, the Estrada index for $f(x) = e^x$ and the graph energy (see \cite{li2012graphenergy} and the references therein) for $f(x) = |x|$. More examples will be discussed at a later stage. If we assume that $f$ has a power series expansion around $0$ that converges everywhere, i.e.,
\begin{equation}
\label{Eq:f}
f(x)=\sum_{k=0}^{\infty}a_kx^k,
\end{equation}
then $\E_f$ satisfies the relation 
\begin{equation}
\label{Eq:ef}
\E_f(G)=\sum_{i=1}^n\sum_{k=0}^{\infty}a_k\lambda_i^k=\sum_{k=0}^{\infty}a_k\M_k(G).
\end{equation}

Let $\mathbb{T}_D$ denote the set of trees with degree sequence $D$. The class of trees with fixed degree sequence is very popular in extremal graph theory. 
For example, it has been studied with regards to the Wiener index \cite{WangHua2008,WangHua2009} and other distance-based invariants \cite{schm12}, spectral radius and Laplacian spectral radius \cite{biyikoglu2008graphs,zhang08b,biyi09},
the energy and the number of independent subsets \cite{Andriantiana2012}, and the number of subtrees \cite{zhang12,zhang12b}. 

The greedy tree $G(D)$ is the tree obtained from a ``greedy algorithm'' that we will describe 
in detail in the following section. Roughly speaking, it is obtained by assigning the highest degree in $D$ to the root, the largest degrees that are left to its neighbors, and so on. 

For any degree sequence $D$, we prove that $G(D)$ has maximum $k$-th spectral moment for any $k\geq 0$, and for sufficiently large $k$, it is unique with this property. Consequently, the greedy tree also maximizes $\E_f$ for any $f$ as in \eqref{Eq:f} among all elements of $\mathbb{T}_D$, provided that the coefficients $a_k$ are nonnegative for even $k$ (the odd spectral moments are $0$ for all bipartite graphs, thus in particular for trees). Details of the proof are provided in Section~\ref{Sec:Main}. Furthermore, in Section~\ref{Sec:Diff} we show that if 
two degree sequences $D=(d_1,\dots,d_n)$ and $B=(b_1,\dots,b_n)$ satisfy 
\begin{equation}
 \label{Eq:Major}
\sum_{i=1}^lb_i \leq \sum_{i=1}^ld_i
\end{equation}
for all $1\leq l \leq n$ (i.e., $D$ \emph{majorizes} $B$), then $\M_k(G(B))\leq \M_k(G(D))$ for any $k\geq 0.$ A number of corollaries can be deduced from these results. In particular a
conjecture of Ili\'c and Stevanovi\'c follows as a corollary to our theorems, which reads as follows:
\begin{conj}[Ili\'c/Stevanovi\'c \cite{Conj}]
\label{Conj:1}
For any $k\geq 2$, the Volkmann tree (see Figure~\ref{fig:volkmann}) has maximum spectral moment $\M_{2k}$ among trees of $n$ vertices with maximum degree $\Delta$.
\end{conj}

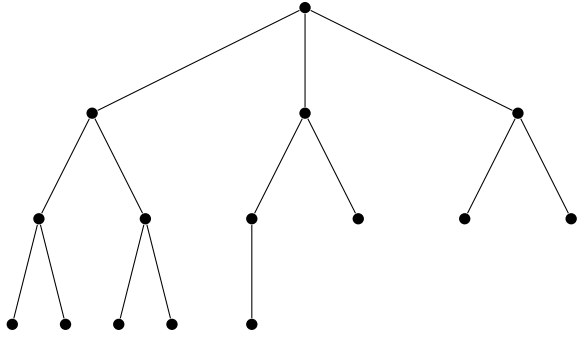
\begin{figure}[htbp]

\centering
    \begin{tikzpicture}[scale=0.7]
        \node[fill=black,circle,inner sep=1.5pt] (v) at (8,6) {};
        \node[fill=black,circle,inner sep=1.5pt] (v1) at (4,4) {};
        \node[fill=black,circle,inner sep=1.5pt] (v2) at (8,4) {};
        \node[fill=black,circle,inner sep=1.5pt] (v3) at (12,4) {};
        \node[fill=black,circle,inner sep=1.5pt] (v11) at (3,2) {};
        \node[fill=black,circle,inner sep=1.5pt] (v12) at (5,2) {};
        \node[fill=black,circle,inner sep=1.5pt] (v21) at (7,2) {};
        \node[fill=black,circle,inner sep=1.5pt] (v22) at (9,2) {};
        \node[fill=black,circle,inner sep=1.5pt] (v31) at (11,2) {};
        \node[fill=black,circle,inner sep=1.5pt] (v32) at (13,2) {};

        \node[fill=black,circle,inner sep=1.5pt] (v111) at (2.5,0) {};
        \node[fill=black,circle,inner sep=1.5pt] (v112) at (3.5,0) {};
        \node[fill=black,circle,inner sep=1.5pt] (v121) at (4.5,0) {};
        \node[fill=black,circle,inner sep=1.5pt] (v122) at (5.5,0) {};
        \node[fill=black,circle,inner sep=1.5pt] (v211) at (7,0) {};

        \draw (v)--(v1);
        \draw (v)--(v2);
        \draw (v)--(v3);
        \draw (v1)--(v11);
        \draw (v1)--(v12);
        \draw (v2)--(v21);
        \draw (v2)--(v22);
        \draw (v3)--(v31);
        \draw (v3)--(v32);
        \draw (v11)--(v111);
        \draw (v11)--(v112);
        \draw (v12)--(v121);
        \draw (v12)--(v122);
        \draw (v21)--(v211);
    \end{tikzpicture}

\caption{The Volkmann tree for $\Delta = 3$, $n = 15$.}
\label{fig:volkmann}
\end{figure}

This, in turn, implies an older conjecture of Gutman, Furtula, Markovi\'c and Gli\v{s}i\'c \cite{gutman2007alkanes}, stating that the Volkmann tree has greatest Estrada index among all trees with maximum degree $\Delta$, see also \cite[p.168]{Gutman2011Estrada}. The Volkmann tree, shown in Figure~\ref{fig:volkmann} in the case $\Delta = 3$ and $n=15$, is essentially a complete $\Delta$-ary tree, and a special case of a greedy tree whose degree sequence is $(\Delta,\Delta,\ldots,\Delta,r,1,1,\ldots,1)$ for some $r$ between $1$ and $\Delta$. Gutman et al. provide an argument supporting their conjecture, which however is not fully rigorous. The Volkmann tree is well-known to be extremal for other graph invariants, notably for the Wiener index \cite{fischermann2002wiener}.

 The conjecture of Ili\'c and Stevanovi\'c was proved by Zhang, Zhou and Li \cite{zhang2011estrada} in the case that the maximum degree $\Delta$ is large (greater than $n/3$). See \cite{du2011estrada1,du2011estrada2,du2012estrada1,du2012estrada2} for further recent extremal results concerning the Estrada index, in particular the Estrada index of trees.

\section{Preliminaries}
\label{Sec:p}
We start with formal definitions of specific terminologies and certain types of trees which will be of central interest in this paper, compare also \cite{schm12,zhang12,andriantiana2013greedy}.
\begin{defn}
Let $F$ be a rooted forest where the maximum height of any component is $k-1$. The \emph{leveled degree sequence} of $F$ is the sequence 
\begin{equation}
\label{Eq:D}
D=(V_1, \dots,V_k),
\end{equation}
where, for any $1\leq i \leq k$, $V_i$ is the non-increasing sequence formed by the degrees of the 
vertices of $F$ at the $i^{\text{th}}$ level (i.e., vertices of distance $i-1$ from the root in the respective component). 
\end{defn}

\begin{defn}
\label{Def:lgf}
The \emph{level greedy forest} with leveled degree sequence 
\begin{equation}\label{eq:ldegseq}
D=((i_{1,1},\dots,i_{1,k_1}), (i_{2,1},\dots,i_{2,k_2}),\dots,(i_{n,1},\dots,i_{n,k_n}))
\end{equation}
is obtained using the following ``greedy algorithm'': 
\begin{enumerate}
\item[(i)] Label the vertices of the first level $g_1^1,\dots,g_{k_1}^1$, and assign degrees 
to these vertices such that $\deg g_j^1 = i_{1,j}$ for all $j$.

\item[(ii)] Assume that the vertices of the $h^{\text{th}}$ level have been labeled 
$g_1^h,\dots,g_{k_h}^h$ and a degree has been assigned to each of them. Then for all 
$1\leq j \leq k_h$ label the neighbors of $g^h_j$ at the $(h+1)^{\text{th}}$ level, if any, 
by $$g_{1+\sum_{m=1}^{j-1}(i_{h,m}-1)}^{h+1},\dots,g^{h+1}_{\sum_{m=1}^{j}(i_{h,m}-1)},$$ and assign 
degrees to the newly labeled vertices such that $\deg g_j^{h+1} = i_{h+1,j}$ for all $j$.
\end{enumerate}
\end{defn}
The level greedy forest with leveled degree sequence $D$ is denoted by $G(D)$. We will use the labeling described in the definition throughout this paper, for level greedy trees and forests as well as related trees.
\begin{defn}
\label{Def:gt}
A connected level greedy forest is called a \emph{level greedy tree}.
\end{defn}
We will also encounter an edge-rooted version of the level greedy tree.
\begin{defn}
The \emph{edge-rooted level greedy tree} with leveled degree sequence
$$D=((i_{1,1},i_{1,2}), (i_{2,1},\dots,i_{2,k_2}),\dots,(i_{n,1},\dots,i_{n,k_n}))$$
is obtained from the two-component level greedy forest with leveled degree sequence
$$((i_{1,1}-1,i_{1,2}-1), (i_{2,1},\dots,i_{2,k_2}),\dots,(i_{n,1},\dots,i_{n,k_n}))$$
by joining the two roots.
\end{defn}
Now, we are ready to define greedy trees:
\begin{defn}
\label{Def:gf}
If a root in a tree can be chosen such that it becomes a level greedy tree whose 
leveled degree sequence, as given in \eqref{eq:ldegseq},
satisfies 
$$\min (i_{j,1},\dots,i_{j,k_j})\geq \max (i_{j+1,1},\dots,i_{j+1,k_{j+1}})$$
for all $1\leq j\leq n-1$, then it is called a \emph{greedy tree}.
\end{defn}

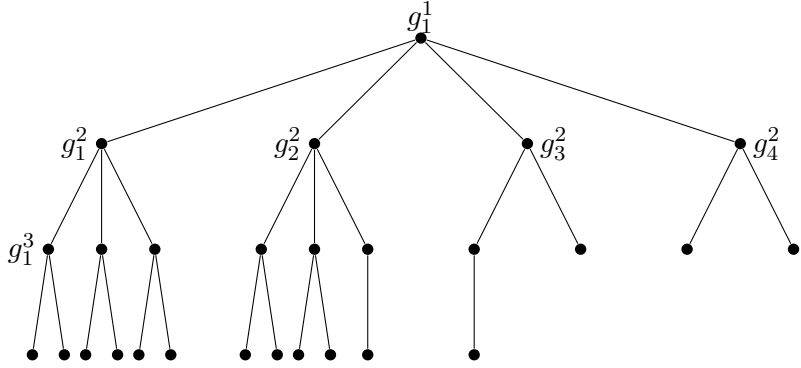
\begin{figure}[htbp]

\centering
    \begin{tikzpicture}[scale=0.7]
        \node[fill=black,circle,inner sep=1.5pt] (v) at (10,6) {};
        \node[fill=black,circle,inner sep=1.5pt] (v1) at (4,4) {};
        \node[fill=black,circle,inner sep=1.5pt] (v2) at (8,4) {};
        \node[fill=black,circle,inner sep=1.5pt] (v3) at (12,4) {};
        \node[fill=black,circle,inner sep=1.5pt] (v4) at (16,4) {};
        \node[fill=black,circle,inner sep=1.5pt] (v42) at (17,2) {};
        \node[fill=black,circle,inner sep=1.5pt] (v12) at (4,2) {};
        \node[fill=black,circle,inner sep=1.5pt] (v41) at (15,2) {};
        \node[fill=black,circle,inner sep=1.5pt] (v32) at (13,2) {};
        \node[fill=black,circle,inner sep=1.5pt] (v22) at (8,2) {};
        \node[fill=black,circle,inner sep=1.5pt] (v31) at (11,2) {};
        \node[fill=black,circle,inner sep=1.5pt] (v23) at (9,2) {};
        \node[fill=black,circle,inner sep=1.5pt] (v21) at (7,2) {};
        \node[fill=black,circle,inner sep=1.5pt] (v13) at (5,2) {};
        \node[fill=black,circle,inner sep=1.5pt] (v11) at (3,2) {};

        \node[fill=black,circle,inner sep=1.5pt] (v111) at (3.3,0) {};
        \node[fill=black,circle,inner sep=1.5pt] (v112) at (2.7,0) {};
        \node[fill=black,circle,inner sep=1.5pt] (v121) at (4.3,0) {};
        \node[fill=black,circle,inner sep=1.5pt] (v122) at (3.7,0) {};
        \node[fill=black,circle,inner sep=1.5pt] (v131) at (5.3,0) {};
        \node[fill=black,circle,inner sep=1.5pt] (v132) at (4.7,0) {};
        \node[fill=black,circle,inner sep=1.5pt] (v211) at (7.3,0) {};
        \node[fill=black,circle,inner sep=1.5pt] (v212) at (6.7,0) {};
        \node[fill=black,circle,inner sep=1.5pt] (v221) at (8.3,0) {};
        \node[fill=black,circle,inner sep=1.5pt] (v222) at (7.7,0) {};
        \node[fill=black,circle,inner sep=1.5pt] (v231) at (9,0) {};
        \node[fill=black,circle,inner sep=1.5pt] (v311) at (11,0) {};

        \draw (v)--(v1);
        \draw (v)--(v2);
        \draw (v)--(v3);
        \draw (v)--(v4);
        \draw (v1)--(v11);
        \draw (v1)--(v12);
        \draw (v1)--(v13);
        \draw (v2)--(v21);
        \draw (v2)--(v22);
        \draw (v2)--(v23);
        \draw (v3)--(v31);
        \draw (v3)--(v32);
        \draw (v4)--(v41);
        \draw (v4)--(v42);
        \draw (v11)--(v111);
        \draw (v11)--(v112);
        \draw (v12)--(v121);
        \draw (v12)--(v122);
        \draw (v13)--(v131);
        \draw (v13)--(v132);
        \draw (v21)--(v211);
        \draw (v21)--(v212);
        \draw (v22)--(v221);
        \draw (v22)--(v222);
        \draw (v23)--(v231);
        \draw (v31)--(v311);

\node at  (10,6.4) {$g_1^1$};

\node at  (3.5,4) {$g_1^2$};
\node at  (7.5,4) {$g_2^2$};
\node at  (12.5,4) {$g_3^2$};
\node at  (16.5,4) {$g_4^2$};

\node at  (2.5,2) {$g_1^3$};
    \end{tikzpicture}

\caption{A greedy tree (only the labels of the first six vertices are shown).}
\label{fig:greedy_tree}
\end{figure}
We denote the set of all permutations of $\{1,\dots,n\}$ by $\mathcal{S}_n$. Let $A=(a_1,\dots,a_n)$ and $B=(b_1,\dots,b_n)$ be sequences of nonnegative numbers. 
We say that $A$ \emph{majorizes} $B$ if for all $1\leq k \leq n$ we have
$$
\sum_{i=1}^ka_i \geq \sum_{i=1}^kb_i.
$$
If for any $\sigma\in \mathcal{S}_n$ the sequence $A$ majorizes 
$(b_{\sigma(1)},\dots,b_{\sigma(n)})$, then we write
\begin{equation}
 \label{Eq:maj}
B \preccurlyeq A.
\end{equation}
\begin{rem}
\label{Rem:1}
Let $\sigma \in \mathcal{S}_n$ be such that $b_{\sigma(1)}\geq\dots \geq b_{\sigma(n)}$.
It is easy to see that $(b_{\sigma'(1)},\dots, b_{\sigma'(n)})\preccurlyeq (b_{\sigma(1)},\dots, b_{\sigma(n)})$
for any $\sigma'\in \mathcal{S}_n$. Relation \eqref{Eq:maj} is equivalent to the statement that $A$ majorizes $(b_{\sigma(1)},\dots,b_{\sigma(n)})$. Furthermore, \eqref{Eq:maj} is equivalent to the statement that for any $k\in\{1,\dots,n\}$
we have 
$$
(b_{\sigma'(1)},\dots,b_{\sigma'(k)}) \preccurlyeq (a_1,\dots,a_k)
$$
for all $\sigma'\in \mathcal{S}_n$.

\end{rem}
The rest of this section consists of a series of lemmas describing properties of 
sequences, which will then be applied to degree sequences in the following sections.
\begin{lem}[cf. \cite{schm12}]
\label{Lem:mj1}
Suppose that $(b_1,\dots,b_n)\preccurlyeq(a_1,\dots,a_n)$ and 
$(b'_1,\dots,b'_n)\preccurlyeq(a'_1,\dots,a'_n)$. Then we have
$$
b'_1b_1+\dots+b'_nb_n\leq a'_1a_1+\dots+a'_na_n.
$$
\end{lem}
The next, stronger looking, lemma is in fact equivalent to Lemma~\ref{Lem:mj1}.
\begin{lem}
 \label{Lem:mj2}
Suppose that $(b_1,\dots,b_n)\preccurlyeq(a_1,\dots,a_n)$ and 
$(b'_1,\dots,b'_n)\preccurlyeq(a'_1,\dots,a'_n)$. Then we have
$$
(b'_1b_1,\dots,b'_nb_n)\preccurlyeq(a'_1a_1,\dots,a'_na_n).
$$
\end{lem}
\begin{proof}
Let $\sigma$ be the element of $\mathcal{S}_n$ for which $b_{\sigma(1)}b'_{\sigma(1)}\geq\dots \geq b_{\sigma(n)}b'_{\sigma(n)}$.
Using Remark~\ref{Rem:1}, we know that for any $k\in \{1,\dots,n\}$, we have 
$$(b_{\sigma(1)},\dots,b_{\sigma(k)})\preccurlyeq(a_1,\dots,a_k)$$
 and 
$$(b'_{\sigma(1)},\dots,b'_{\sigma(k)})\preccurlyeq(a'_1,\dots,a'_k).$$
 By Lemma~\ref{Lem:mj1} this implies 
$$b_{\sigma(1)}b_{\sigma(1)}'+\dots+b_{\sigma(k)}b'_{\sigma(k)}\leq a_1a'_1+\dots+a_ka'_k.$$
Hence, $(a_1a'_1,\dots,a_na'_n)$ majorizes $(b_{\sigma(1)}b'_{\sigma(1)},\dots, b_{\sigma(n)}b'_{\sigma(n)})$, 
and the lemma follows from Remark~\ref{Rem:1}.
\end{proof}
Let $(k_1,\dots,k_n)$ be a sequence of integers. For any sequence $(a_1,\dots,a_n)$, we define 
$$
(a_1,\dots,a_n)*(k_1,\dots,k_n)=\big(b_1,\dots,b_{\sum_{i=1}^nk_i}\big),
$$
where $b_j=a_{\ell}$ whenever $\sum_{i=1}^{\ell-1}k_i< j \leq \sum_{i=1}^{\ell}k_i$ (i.e., each $a_{\ell}$ is repeated $k_{\ell}$ times). For example, $(1,3,2)*(2,3,4)=(1,1,3,3,3,2,2,2,2).$
\begin{rem}
\label{Rem:2}
 It is easy to see that if the sequences $(k_1,\dots,k_n)$ and $(a_1,\dots,a_n)$ are non-increasing,
then for any $\sigma$ and $\pi$ in $\mathcal{S}_n$ we have
$$
(a_{\sigma(1)},\dots,a_{\sigma(n)})*(k_{\pi(1)},\dots,k_{\pi(n)}) \preccurlyeq (a_1,\dots,a_n)*(k_1,\dots,k_n).
$$
\end{rem}

\begin{lem}
\label{Lem:mj3}
Assume that $B=(b_1,\dots,b_n)\preccurlyeq(a_1,\dots,a_n)=A$ and let $C=(c_1,\dots,c_n)$ be a 
non-increasing sequence of positive integers. Then for any $\sigma \in \mathcal{S}_n$ we have
$
B*(c_{\sigma(1)},\dots,c_{\sigma(n)}) \preccurlyeq A*C.
$
\end{lem}
\begin{proof}
Let $\pi\in\mathcal{S}_n$ be such that $b_{\pi(1)}\geq \dots \geq b_{\pi(n)}$, and let 
$B_{\pi}=(b_{\pi(1)}, \dots,$ $b_{\pi(n)})$. By Remark~\ref{Rem:2}, we know that 
$B*(c_{\sigma(1)},\dots,c_{\sigma(n)}) \preccurlyeq B_{\pi}*C.$ Since $B_{\pi}*C$ is a 
non-increasing sequence, we can prove the lemma by showing that $A*C$ majorizes $B_{\pi}*C$.

The case $n=1$ is trivial. Assume that the statement holds for $n=k$. For $n=k+1$, the relation $B\preccurlyeq A$ implies that 
$(b_{\pi(1)},\dots,b_{\pi(k)})\preccurlyeq(a_1,\dots,a_{k})$. By the induction hypothesis we deduce that 
\begin{equation}
 \label{Eq:8}
(b_{\pi(1)},\dots,b_{\pi(k)})*(c_1,\dots,c_k)\preccurlyeq(a_1,\dots,a_k)*(c_1,\dots,c_k).
\end{equation}
Now we reason by induction with respect to $c_{k+1}$. For any two sequences $S=(s_1,\dots,s_l)$ 
and $S'=(s'_1,\dots,s'_{l'})$, let $S:S'$ denote the sequence obtained by concatenation, i.e., $(s_1,\dots,s_l,s'_1,$ $\dots,s'_{l'})$.
If $c_{k+1}=1$, then $$(b_{\pi(1)},\dots,b_{\pi(k+1)})*C=((b_{\pi(1)},\dots,b_{\pi(k)})*(c_1,\dots,c_k)):(b_{\pi(k+1)})$$ and 
$A*C=((a_1,\dots,a_k)*(c_1,\dots,c_k)):(a_{k+1})$. Using Lemma~\ref{Lem:mj1} we know that
$$
\ssum (B_{\pi}*C)=\sum_{i=1}^{k+1}b_{\pi(i)}c_i\leq \sum_{i=1}^{k+1}a_ic_i =\ssum (A*C),
$$
where $\ssum (B_{\pi}*C)$ and $\ssum (A*C)$ are the sums of the entries in $B_{\pi}*C$ 
and $A*C$, respectively. With \eqref{Eq:8}, this implies that $A*C$ majorizes $(b_{\pi(1)},\dots,b_{\pi(k+1)})*C$.
 The (second) induction step follows from the relations
\begin{align*}
B_{\pi}*(c_1,\dots,c_{k+1})&=(b_{\pi(1)},\dots,b_{\pi(k+1)})*(c_1,\dots,c_{k+1}-1):(b_{\pi(k+1)}),\\
A*(c_1,\dots,c_{k+1})&=(a_1,\dots,a_{k+1})*(c_1,\dots,c_{k+1}-1):(a_{k+1}).
\end{align*}
\end{proof}

\section{Trees with given degree sequence}
\label{Sec:Main}
Let $T$ be a tree and $v$ one of its vertices. We denote by $\mathcal{W}_v(k;T)$ the set of all walks of length $k$ in $T$ starting at $v$, and by $\mathcal{C}_v(k;T)$ the set of all closed walks of length $k$ in $T$ starting and ending at $v$. We also write
\begin{equation}
\label{Eq:op}
\mathcal{W}(k;T)=\bigcup_{v\in V(T)}\mathcal{W}_v(k;T)
\end{equation}
for the set of all walks of length $k$ in $T$ and
\begin{equation}
\label{Eq:cl}
\mathcal{C}(k;T)=\bigcup_{v\in V(T)}\mathcal{C}_v(k;T)
\end{equation}
for the set of all closed walks of length $k$. Note that $\mathcal{C}(k;T) = \emptyset$ whenever $k$ is odd.

\subsection{Vertex rooted trees}
\label{Sec:vr}
Let $W=w_1\dots w_k$ be a walk in a rooted tree $T$. We say that $(i_1,i_2,\ldots,i_k)$ is the \emph{level sequence} of $W$ if $w_l$ is at the $i_l^{\text{th}}$ 
level in $T$, i.e., at distance $i_l-1$ from the root, for all $l \leq k$. We denote by $\mathcal{W}(i_1,\dots,i_k;T)$ the set of walks with level sequence 
$(i_1,\dots,i_k)$ in $T$. For any vertex $v$ of $T$ we define
$$
\mathcal{W}_{v}(i_1,\dots,i_k;T)=\{w_1\dots w_k\in \mathcal{W}(i_1,\dots,i_k;T): w_1=v\}.
$$
The sets $\mathcal{C}(i_1,\dots,i_k;T)$ and $\mathcal{C}_{v}(i_1,\dots,i_k;T)$ are defined analogously. Moreover, we denote the cardinalities of $\mathcal{W}(k;T)$ and $\mathcal{C}(k;T)$ by $W(k;T)$ and $C(k;T)$ respectively, the cardinality of $\mathcal{W}_{v}(i_1,\dots,i_k;T)$ by $W_{v}(i_1,\dots,i_k;T)$, etc. This convention will be kept even if not mentioned explicitly. Finally, the set of rooted forests with leveled degree sequence $D$ is denoted by $\mathcal{T}_D$.
\begin{lem}
 \label{Lem:open_wa_ver}
Let $T\in \mathcal{T}_D$ for some leveled degree sequence $D$ of a vertex-rooted forest, and let $G = G(D)$ be the associated greedy forest. Let 
$v_1^i,\dots,v_{d_i}^i$ be the vertices of $T$ at the $i^{\text{th}}$ level. 
Then for any level sequence of walks $(i_1,\dots,i_l)$, the following relations hold for all $i$:
\begin{align}
\label{Eq:th1}
 &(W_{v^{i}_1}(i_1,\dots,i_l;T),\dots,W_{v^{i}_{d_{i}}}(i_1,\dots,i_l;T))\preccurlyeq (W_{g^{i}_1}(i_1,\dots,i_l;G),\dots,W_{g^{i}_{d_{i}}}(i_1,\dots,i_l;G))
\end{align}
and
\begin{equation}
 \label{Eq:th11}
W_{g^{i}_1}(i_1,\dots,i_l;G)\geq W_{g^{i}_2}(i_1,\dots,i_l;G)\geq \dots\geq W_{g^{i}_{d_{i}}}(i,\dots,i_l;G).
\end{equation}
\end{lem}
\begin{proof}
The situation where $i\neq i_1$ is not interesting, since we get 
$$W_{v^{i}_j}(i_1,\dots,i_l;T)=W_{g^{i}_j}(i_1,\dots,i_l;G)=0$$
for any $j$. So we assume that $i=i_1$ and proceed by induction with respect to $l$.
The initial case $l=1$ is trivial, since we know that 
$$W_{v^{i_1}_j}(i_1;T)=W_{g^{i_1}_j}(i_1;G)=1$$
for all $i_1$ and $j$. Assume that the relations \eqref{Eq:th1} and \eqref{Eq:th11} hold whenever $l\leq k$ for some integer $k\geq 1$. Now consider a longer level sequence $(i_1,\dots,i_l)$ where $l=k+1$. There are two cases: $i_2 = i_1 - 1$ or $i_2 = i_1 + 1$ (in all other cases, the number of walks is $0$).
\begin{itemize}
 \item \textbf{Case 1:} Assume that $i_2 = i_1+1 = i+1$.
For $1\leq j \leq d_{i}$, we use $a_j$ as an abbreviation for the number of children of $v^i_j$ and $b_j$ for the number of children of $g^i_j$. Clearly, $a_j = \deg v_j^{i} -1$ and $b_j = \deg g_j^i - 1$ if $i \neq 1$, and $a_j = \deg v_j^i$, $b_j = \deg g_j^i$ if $i =1$. In view of the construction of greedy trees, we have
$$b_1 \geq b_2 \geq \cdots \geq b_{d_i},$$
and since $(a_1,a_2,\ldots,a_{d_i})$ is a permutation of $(b_1,b_2,\ldots,b_{d_i})$, it is clear that
\begin{equation}\label{Eq:Deg}
(a_1,a_2,\ldots,a_{d_i})  \preccurlyeq (b_1,b_2,\ldots,b_{d_i}).
\end{equation}
We also write $r_j$ and $s_j$ for the sums
$$r_j = \sum_{t=1}^j a_t \quad \text{and} \quad s_j = \sum_{t=1}^j b_t,$$
and $r_0 = s_0 = 0$. Now note that
$$
W_{v^{i}_j}(i_1,\dots,i_l;T) = \sum_{v^{i+1}_h \sim v_j^{i}} W_{v^{i+1}_h}(i_2,\dots,i_l;T) =
\sum_{h = r_{j-1}+1}^{r_j} W_{v^{i+1}_h}(i_2,\dots,i_l;T),
$$
since every walk with level sequence $(i_1,i_2,\ldots,i_l)$ starting at $v^i_j$ has to go to one of the children $v^{i+1}_h$ ($r_{j-1}+1 \leq h \leq r_j$) first.
Likewise,
$$
W_{g^{i}_j}(i_1,\dots,i_l;G) = \sum_{g^{i+1}_h \sim g_j^{i}} W_{g^{i+1}_h}(i_2,\dots,i_l;G) =
\sum_{h = s_{j-1}+1}^{s_j} W_{g^{i+1}_h}(i_2,\dots,i_l;G).
$$
Now the relation 
$$(W_{v^{i}_1}(i_1,\dots,i_l;T),\dots,W_{v^{i}_{d_{i}}}(i_1,\dots,i_l;T)) 
\preccurlyeq  (W_{g^{i}_1}(i_1,\dots,i_l;G),\dots,W_{g^{i}_{d_{i}}}(i_1,\dots,i_l;G))$$
as well as \eqref{Eq:th11}, follow from the induction hypothesis applied to the level sequence $(i_2,\dots,i_l)$ and the majorization inequality \eqref{Eq:Deg}, which also implies that $r_j \leq s_j$ for all $j$.

\item \textbf{Case 2:} Assume that $i_2 = i_1 - 1 = i-1$. This time, we write $a_j$ for the number of children of $v^{i-1}_j$ (which is either $\deg v^{i-1}_j$ or $\deg v^{i-1}_j - 1$) and $b_j$ for the number of children of $g^{i-1}_j$. The relation~\eqref{Eq:Deg} is still valid. Now we have
\begin{multline*}
(W_{v^{i}_1}(i_1,\dots,i_l;T),\dots,W_{v^{i}_{d_{i}}}(i_1,\dots,i_l;T))\\
= (W_{v^{i-1}_1}(i_2,\dots,i_l;T),\dots,W_{v^{i-1}_{d_{i-1}}}(i_2,\dots,i_l;T))*
  (a_1,\ldots,a_{d_{i-1}}),
\end{multline*}
since if $v^i_h$ is one of the $a_j$ children of $v^{i-1}_j$, a walk with level sequence $(i_1,\ldots,i_l)$ starting at $v^i_h$ has to start with a step to $v^{i-1}_j$, which means that
$$W_{v^i_h}(i_1,\dots,i_l;T) = W_{v^{i-1}_j}(i_2,\dots,i_l;T).$$
Likewise,
\begin{multline*}
 (W_{g^{i}_1}(i_1,\dots,i_l;G),\dots,W_{g^{i}_{d_{i}}}(i_1,\dots,i_l;G))\\
= (W_{g^{i-1}_1}(i_2,\dots,i_l;G),\dots,W_{g^{i-1}_{d_{i-1}}}(i_2,\dots,i_l;G))*
  (b_1,\ldots,b_{d_{i-1}}).
\end{multline*}
So \eqref{Eq:th1} and \eqref{Eq:th11} follow from \eqref{Eq:Deg} and the induction hypothesis by means of Lemma~\ref{Lem:mj3}. 
\end{itemize}
\end{proof}

Next we study closed walks: it turns out that a completely analogous statement holds.
\begin{lem}
\label{Lem:closed_wa_ver}
Let $T\in \mathcal{T}_D$ for some leveled degree sequence $D$ of a vertex-rooted forest, and let $G = G(D)$ be the associated greedy forest. Let 
$v_1^i,\dots,v_{d_i}^i$ be the vertices of $T$ at the $i^{\text{th}}$ level. 
Then for any level sequence of walks $(i_1,\dots,i_l)$, the following relations hold for all $i$:
\begin{align}
\label{Eq:th2}
 &(C_{v^{i}_1}(i_1,\dots,i_l;T),\dots,C_{v^{i}_{d_{i}}}(i_1,\dots,i_l;T))\preccurlyeq (C_{g^{i}_1}(i_1,\dots,i_l;G),\dots,C_{g^{i}_{d_{i}}}(i_1,\dots,i_l;G))
\end{align}
and
\begin{equation}
 \label{Eq:th22}
C_{g^{i}_1}(i_1,\dots,i_l;G)\geq\dots\geq C_{g^{i}_{d_{i}}}(i,\dots,i_l;G).
\end{equation}
\end{lem}
\begin{proof}
As in the proof of Lemma~\ref{Lem:open_wa_ver}, we only need to prove the lemma for $i=i_1$.
The case when $l$ is even is trivial: in this case,
\begin{equation*}
\mathcal{C}_{v^{i}_j}(i_1,\dots,i_l;T)=\mathcal{C}_{g^{i}_j}(i_1,\dots,i_l;G)=\emptyset
\end{equation*}
for all $j$, since there are no closed walks of odd length in a forest. 

For the case of odd $l$, say $l=2l'-1$, the proof is similar to that of Lemma~\ref{Lem:open_wa_ver}:
We reason by induction with respect to $l'$. The case $l'=1$  is again trivial. Assume that the lemma holds 
for all $l'\leq k$ for some $k\geq 1$. Now consider a level sequence $(i_1,\dots,i_{2k+1})$. 
We must have $i_1 = i_{2k+1} = i$ and $i_2 = i \pm 1$ as well as $i_{2k} = i \pm 1$, the other possibilities are trivial.

\medskip

\noindent \textbf{Case 1:} If $i_2 = i_{2k} = i-1$, then, writing $a_j$ for the number of children of $v^{i-1}_j$ and $b_j$ for the number of children of $g^{i-1}_j$, we have
\begin{align*}
&(C_{v^{i}_1}(i_1,\dots,i_{2k+1};T),\dots,C_{v^{i}_{d_{i}}}(i_1,\dots,i_{2k+1};T)) \\
&= (C_{v^{i-1}_1}(i_2,\dots,i_{2k};T),\dots,C_{v^{i-1}_{d_{i-1}}}(i_2,\dots,i_{2k};T))*(a_1,\dots,a_{d_{i-1}})
\end{align*}
and
\begin{align*}
&(C_{g^{i}_1}(i_1,\dots,i_{2k+1};G),\dots,C_{g^{i}_{d_{i}}}(i_1,\dots,i_{2k+1};G)) \\
&= (C_{g^{i-1}_1}(i_2,\dots,i_{2k};G),\dots,C_{g^{i-1}_{d_{i-1}}}(i_2,\dots,i_{2k};G))*(b_1,\dots,b_{d_{i-1}})
\end{align*}
for the same reason as in Case 2 of Lemma~\ref{Lem:open_wa_ver}.
Hence \eqref{Eq:th2} and \eqref{Eq:th22} can be 
obtained using Lemma~\ref{Lem:mj3}, the induction hypothesis and \eqref{Eq:Deg}.

\medskip

\noindent \textbf{Case 2:} Assume that $i_2 = i_1+1 = i+1$. Let $h$ be the smallest integer such that $h > 1$ and $i_h=i_1 = i$. If $h$ does not exist, then there is no closed walk with level sequence $(i_1,\dots,i_{2k+1})$, so we can ignore this case. By the definition of $h$ and the assumption that $i_2 = i+1$, we know that $i=\min (i_1,\dots,i_h)$. Clearly, 
any walk with level sequence $(i_1,\dots,i_h)$ is closed. Hence for all $j$, any element of 
$\mathcal{C}_{g^{i}_j}(i_1,\dots,i_{2k+1};G)$ can be decomposed (uniquely) into a first part that is an element of $\mathcal{C}_{g^{i}_j}(i_1,\dots,i_h;G)$ and a second part that is an element of $\mathcal{C}_{g^{i}_j}(i_h,\dots,i_{2k+1};G)$.
Similarly, an element of 
$\mathcal{C}_{v^{i}_j}(i_1,\dots,i_{2k+1};T)$ splits (uniquely) into two parts:
a first part in $\mathcal{C}_{v^{i}_j}(i_1,\dots,i_h;T)$ and a second part in
$\mathcal{C}_{v^{i}_j}(i_h,\dots,i_{2k+1};T)$. This implies that 
\begin{align*}
C_{g^{i}_j}(i_1,\dots,i_{2k+1};G)
&= C_{g^{i}_j}(i_1,\dots,i_{h};G)C_{g^{i}_j}(i_h,\dots,i_{2k+1};G)\\
&= W_{g^{i}_j}(i_1,\dots,i_{h};G)C_{g^{i}_j}(i_h,\dots,i_{2k+1};G)
\end{align*}
and
\begin{align*}
C_{v^{i}_j}(i_1,\dots,i_{2k+1};T)
&= C_{v^{i}_j}(i_1,\dots,i_{h};T)C_{v^{i}_j}(i_h,\dots,i_{2k+1};T)\\
&= W_{v^{i}_j}(i_1,\dots,i_{h};T)C_{v^{i}_j}(i_h,\dots,i_{2k+1};T).
\end{align*}
Therefore, we can use Lemma~\ref{Lem:open_wa_ver}, the induction hypothesis and Lemma~\ref{Lem:mj2} 
to deduce \eqref{Eq:th2} and \eqref{Eq:th22}; the argument remains valid even if $h=2k+1$, since then the second factor in the formulas above is simply $1$.

\medskip

\noindent \textbf{Case 3:} Assume that $i_{2k} = i_{2k+1}+1 = i+1$. Then the sequence $(i_{2k+1},\dots,i_1)$ satisfies the condition of Case 2. Hence, for this case, \eqref{Eq:th2} and \eqref{Eq:th22} follow from the fact that for any $j$ we have
$$C_{g^{i}_j}(i_1,\dots,i_{2k+1};G) = C_{g^{i}_j}(i_{2k+1},\dots,i_1;G)$$
and 
$$C_{v^{i}_j}(i_1,\dots,i_{2k+1};T) = C_{v^{i}_j}(i_{2k+1},\dots,i_1;T).$$
This completes the proof, since  there are no closed walks in any other cases.
\end{proof}

The following theorem is a direct consequence of the two Lemmas~\ref{Lem:open_wa_ver} and
\ref{Lem:closed_wa_ver} and the relations \eqref{Eq:op} and \eqref{Eq:cl}.

\begin{thm}
\label{Thm:Main_ver_root}
Let $D$ be a leveled degree sequence of a vertex-rooted forest and $G(D)$ the associated level greedy forest. Then for any nonnegative integer $k$ and all $T\in \mathcal{T}_D$, we have
$$W(k;T)\leq W(k;G(D))$$
and
$$\M_k(T)=C(k;T)\leq C(k;G(D))=\M_k(G(D)).$$
\end{thm}

It turns out that one has strict inequality for sufficiently large even $k$, which is shown in the following lemma:

\begin{lem}\label{Lem:Strict}
Let $D$ be a leveled degree sequence of a vertex-rooted forest and $G = G(D)$ the associated level greedy forest. If $T \in \mathcal{T}_D$ is not isomorphic (as a rooted forest) to $G$, then there exists an integer $k_0$ such that
$$\M_k(T)=C(k;T) <  C(k;G(D))=\M_k(G(D))$$
for all even $k \geq k_0$.
\end{lem}

\begin{proof}
It suffices to find one specific level sequence for which we have strict inequality. We take $h_2$ to be the smallest positive integer such that $T$, restricted to the first $h_2$ levels, is not isomorphic to a level greedy rooted forest. Then let $h_1$ be the largest positive integer such that the restriction of $T$ to levels $h_1,h_1+1,\ldots,h_2$ (which we denote by $P$) is still not isomorphic to a greedy rooted forest.

From now on, we only work with the restricted forest $P$. Let $r$ be the number of its roots and
$P_1,P_2,\ldots,P_r$ the components of $P$. Each of them is a level greedy tree: if not, we could remove the root to obtain a rooted forest that is not level greedy, contradicting the maximality of $h_1$. However, by assumption, their union is not a level greedy forest.

Now let $p_1,p_2,\ldots,p_r$ be the number of descendants of the $r$ roots at level $h_2$ ($p_j$ descendants in component $P_j$). The analogous numbers for the greedy tree are $q_1,q_2,\ldots,q_r$, and we call the corresponding components of the restriction of $G$ to the same levels $Q_1,Q_2,\ldots,Q_r$.

We assume, without loss of generality, that $p_1 \geq p_2 \geq \cdots \geq p_r$ and $q_1 \geq q_2 \geq \cdots \geq q_r$. From the construction of level greedy forests, we know that
$$(p_1,p_2,\ldots,p_r) \preccurlyeq (q_1,q_2,\ldots,q_r).$$
In fact, this is a special case of Lemma~\ref{Lem:open_wa_ver}, since $p_1,\ldots,p_r$ and $q_1,\ldots,q_r$ also count walks with level sequence $(h_1,h_1+1,\ldots,h_2)$. The number of closed walks with level sequence 
$$(h_2,h_2-1,\ldots,h_1+1,h_1,h_1+1,\ldots,h_2-1,h_2,h_2-1,\ldots,h_1+1,h_1,h_1+1,\ldots,h_2)$$
in $T$ and $G$ are
$$p_1^2 + p_2^2 + \cdots + p_r^2\quad \text{and}\quad q_1^2 + q_2^2 + \cdots + q_r^2$$
respectively: such walks start at level $h_2$, move up to the root, return to level $h_2$, then back to the root, and back to the starting point. They are thus completely determined by the two vertices at level $h_2$ (not necessarily distinct), which have to have the same root.

We suppose first that $p = (p_1,p_2,\ldots,p_r) \neq (q_1,q_2,\ldots,q_r) = q$. Let $i$ be the first index and $j$ the last index where the two differ. Since $q$ majorizes $p$ and the two have the same sums, we must have $q_i > p_i$ and $q_j < p_j$. Let $\epsilon = \min(q_i-p_i,p_j-q_j)$, and replace $p_i$ by $p_i+\epsilon$ and $p_j$ by $p_j-\epsilon$. Then the sum of squares increases by
$$(p_i+\epsilon)^2 - p_i^2 + (p_j-\epsilon)^2 - p_j^2 = 2\epsilon(p_i-p_j+\epsilon) > 0.$$
Repeating this process, we can transform $p$ into $q$, which shows that
$$p_1^2 + p_2^2 + \cdots + p_r^2 < q_1^2 + q_2^2 + \cdots + q_r^2,$$
and we are done in that we have found a level sequence such that $G$ has strictly more closed walks than $T$. The same argument applies (mutatis mutandis) to level sequences of the form
\begin{multline*}
(h_2,h_2-1,\ldots,h_1+1,h_1,h_1+1,h_1,h_1+1,h_1,\ldots,  \\
h_1,h_1+1,\ldots,h_2,h_2-1,\ldots,h_1,h_1+1,\ldots,h_2),
\end{multline*}
completing the proof in the case that $p$ and $q$ are not identical (with $k_0 = 4(h_2-h_1)$).

Let us now assume that $p = (p_1,p_2,\ldots,p_r) = (q_1,q_2,\ldots,q_r) = q$, and let $l$ be the last index such that $p_l = q_l \neq 0$. By our choice of $h_2$, the restrictions of $T$ and $G$ to levels $h_1,h_1+1,\ldots,h_2-1$ are isomorphic: they are both level greedy forests consisting of $r$ components. If one component is larger than another, then the number of vertices at level $h_2-1$ is greater as well, and if two components have the same number of vertices at level $h_2-1$, then they are isomorphic by the construction of greedy trees.

Let $m$ be the number of vertices at level $h_2-1$ in the largest component. Then $q_1$ is the sum of the highest $m$ degrees at level $h_2-1$. The only way how $p_1$ can be equal to $q_1$ is thus that $P_1$ and $Q_1$ have the same number of vertices at level $h_2-1$, so they have to be isomorphic (both are known to be level greedy as well!). Likewise, $P_2$ and $Q_2$ have to be isomorphic, etc. The only possible exception are $P_l$ and $Q_l$, the last components with vertices at level $h_2$: here, some vertices in $Q_l$ at level $h_2-1$ might be leaves, so $P_l$ could be smaller than $Q_l$.

Now let $p_1',p_2',\ldots,p_r'$ and $q_1',q_2',\ldots,q_r'$ be the number of vertices at level $h_2-1$ in $P_1,P_2,\ldots,P_r$ and $Q_1,Q_2,\ldots,Q_r$ respectively. The number of closed walks with level sequence
$$(h_2-1,\ldots,h_1+1,h_1,h_1+1,\ldots,h_2-1,h_2,h_2-1,\ldots,h_1+1,h_1,h_1+1,\ldots,h_2-1)$$
in $T$ and $G$ are
$$p_1p_1' + p_2p_2' + \cdots + p_rp_r'\quad \text{and}\quad q_1q_1' + q_2q_2' + \cdots + q_rq_r'$$
respectively, by the same reasoning as before. We know that $p_ip_i' = q_iq_i'$ for $i < l$ and $p_ip_i' = q_iq_i' = 0$ for $i > l$, thus the difference between the two is
$$(q_1q_1' + q_2q_2' + \cdots + q_rq_r') - (p_1p_1' + p_2p_2' + \cdots + p_rp_r') = q_lq_l' - p_lp_l' = q_l(q_l'-p_l').$$
If $q_l' = p_l'$, then the components $P_l$ and $Q_l$ up to level $h_2-1$ have to be isomorphic, and since both are level greedy up to level $h_2$ as well, they must be isomorphic. But then $T$ and $G$, restricted to levels $h_1,h_1+1,\ldots,h_2$, are isomorphic, contradicting our choice of $h_1$ and $h_2$. Thus $q_l' > p_l'$, which means that we have again found a suitable level sequence. Once again, one can generalize to
\begin{multline*}
(h_2-1,\ldots,h_1+1,h_1,h_1+1,h_1,h_1+1,h_1,\ldots,  \\
h_1,h_1+1,\ldots,h_2,h_2-1,\ldots,h_1,h_1+1,\ldots,h_2-1),
\end{multline*}
to show that we have strict inequality $C(k;T) < C(k;G)$ for all even $k \geq k_0$, now with $k_0 = 4(h_2-h_1)-2$.
\end{proof}

\subsection{Edge rooted trees}
\label{Sec:edge_r_t}
As we will see at the end of this subsection, Theorem~\ref{Thm:Main_ver_root} still holds if we 
consider edge-rooted trees instead of vertex-rooted trees.

For any set $\mathcal{A}$ of walks in a graph and any vertex $v$ and edge $e$ of the same graph, we denote by $\mathcal{A}^e$ and by $\mathcal{A}^v$ the subsets of $\mathcal{A}$ that only contain walks passing through $e$ and $v$, respectively.
Instead of $(\mathcal{A}^e)^{e'}$ we simply write $\mathcal{A}^{e,e'}$. Similarly, $(\mathcal{A}^e)^v=(\mathcal{A}^v)^e=\mathcal{A}^{v,e}=\mathcal{A}^{e,v}.$
For any two adjacent vertices $u$ and $v$ in a graph $G$, we define $\mathcal{C}_{u,v}(k;G)$ to be the set and $C_{u,v}(k;G)$ the number of all closed 
walks of length $k$ starting from the edge $uv$ in direction from $u$ to $v$.

Different combinations of these notations are possible. For example, for some edge $uv$ in a graph $G$ and another edge $e$, $\mathcal{C}_{u,v}^{e}(k;G)$ stands for the set of closed walks of length $k$ in $G$ starting at $u$, using the edge $uv$ at the first step and passing through $e$ at a later stage. 

\begin{lem}
\label{Lem:Adj_uv}
Let $u$ and $v$ be two adjacent vertices in a graph $G$, and let $e$ be an edge in $G$. Then for all nonnegative integers $k$ we 
have 
$$
C_{u,v}(k;G)=C_{v,u}(k;G)
\quad\text{and}\quad
C^e_{u,v}(k;G)=C^e_{v,u}(k;G).
$$
\end{lem}
\begin{proof}
Both $C_{u,v}(k;G)$ and $C_{v,u}(k;G)$ are equal to the number of walks of length $k-1$ starting from $u$ and ending at $v$ (which is clearly the same as the number of walks of length $k-1$ starting from $v$ and ending at $u$).

If $e\neq uv$, then both $C^e_{u,v}(k;G)$ and 
$C^e_{v,u}(k;G)$ are equal to the number of walks of length $k-1$ starting 
from $u$, passing through $e$ and ending at $v$. If $e=uv$, then clearly $C^{e}_{u,v}(k;G)=C_{u,v}(k;G)$, and we are done.
\end{proof}

For any (edge- or vertex-) rooted tree $T$ we denote by $\ro(T)$ the root of $T$. We extend the 
notation $\mathcal{C}_{v}(k;T)$ and denote by $\mathcal{C}_{e}(k;T)$ the set of walks of 
length $k$ in $T$ which start with the edge $e$ (in either direction). As usual, $C_{v}(k;T)$ and $C_{e}(k;T)$ denote their cardinalities. If $T$ is an edge-rooted tree such that $u$ and 
$v$ are the ends of $\ro(T)$, we know by Lemma~\ref{Lem:Adj_uv} that
$$
C_{\ro(T)}(k;T) = C_{u,v}(k;T) + C_{v,u}(k;T)=2 C_{u,v}(k;T)=2 C_{v,u}(k;T).
$$
\begin{lem}
\label{Lem:Start_Root}
Let $D$ be a leveled degree sequence of an edge-rooted tree and $G = G(D)$ the associated edge-rooted greedy tree. For any element $T\in \mathbb{T}_D$ we have
$$C_{\ro(T)}(k;T)\leq C_{\ro(G)}(k;G)$$ for any nonnegative integer $k$.
\end{lem}
\begin{proof}
Let $G_1$ and $G_2$ be the components of $G-\ro(G)$, and let $T_1$ and $T_2$ be the 
components of $T-\ro(T)$. Since for odd $k$ we trivially have 
$C_{\ro(T)}(k;T) = C_{\ro(G)}(k;G)=0$, we are only interested in even $k=2l.$
Let us reason by induction on $l$. The cases where $l=1,2$ are easy to check, since the closed walks of 
length at most $4$ starting with the root edge cannot reach beyond the first two levels, but these parts 
of $T$ and $G$ are isomorphic edge-rooted trees. Assume that the lemma holds whenever $l\leq m$ for 
some integer $m\geq 2$. Now consider the case where $l=m+1$. The level sequences of 
the elements in $\mathcal{C}_{\ro(T)}(k;T)$ and $\mathcal{C}_{\ro(G)}(k;G)$ are of the form $(1,1,i_1,i_2,\dots,i_{k-1})$, and $i_{k-1}$ also has to be $1$.

We first consider walks that do not return immediately to the starting point after the first step. For any $j$ with $2 \leq j\leq k-1$, let $\mathcal{C}^j_{\ro(T)}(k;T)$ and $\mathcal{C}^j_{\ro(G)}(k;G)$ be respectively the subsets of $\mathcal{C}_{\ro(T)}(k;T)$ and $\mathcal{C}_{\ro(G)}(k;G)$ whose elements are the walks with level sequences $(1,1,i_1,i_2,\dots,i_{k-1})$, where $i_j=1$ and 
$1\notin \{i_1,i_2,\dots,i_{j-1}\}$. Their cardinalities are denoted by $C^j_{\ro(T)}(k;T)$ and $C^j_{\ro(G)}(k;G)$ respectively. These walks start with the edge root, then go on to higher levels, return to level $1$ for the first time after $j$ steps, and then continue with $k-j-1$ more steps until they return to the starting point. We can uniquely split each of these walks into the $j$ steps from the first step to level $2$ to the first return to level $1$ and the rest. Set
$$\mathbb{S}_j=\{(1,i_1,i_2,\dots,i_{j-1},1): 1\notin \{i_1,i_2,\dots,i_{j-1}\}\}.$$
From Lemma~\ref{Lem:closed_wa_ver}, Lemma~\ref{Lem:Adj_uv} and the induction hypothesis, we now obtain
\begin{align*}
C^j_{\ro(T)}(k;T) &=C_{\ro(T_1),\ro(T_2)}(k-j;T) \sum_{S\in \mathbb{S}_j}C_{\ro(T_2)}(S;T_2)+C_{\ro(T_2),\ro(T_1)}(k-j;T)\sum_{S\in \mathbb{S}_j}C_{\ro(T_1)}(S;T_1)\\
&=\frac{1}{2}C_{\ro(T)}(k-j;T)\big(\sum_{S\in \mathbb{S}_j}C_{\ro(T_2)}(S;T_2)+\sum_{S\in \mathbb{S}_j}C_{\ro(T_1)}(S;T_1)\big)\\
&\leq \frac{1}{2}C_{\ro(G)}(k-j;G)\big(\sum_{S\in \mathbb{S}_j}C_{\ro(G_2)}(S;G_2)+\sum_{S\in \mathbb{S}_j}C_{\ro(G_1)}(S;G_1)\big)\\
&=C^j_{\ro(G)}(k;G)
\end{align*}
for any $j\geq 2$. This covers all the cases where $i_1\neq 1$. Next, consider the subsets 
$\mathcal{C}^*_{\ro(T)}(k;T)$ and $\mathcal{C}^*_{\ro(G)}(k;G)$ of 
$\mathcal{C}_{\ro(T)}(k;T)$ and $\mathcal{C}_{\ro(G)}(k;G)$, respectively; their elements are 
closed walks with level sequence $(1, 1, i_1, i_2,\dots,$ $i_{k-1})$, where $i_1=1$ and for any $h\in \{1,2,\dots,k-2\}$ we always have 
$(1,1)\neq (i_h,i_{h+1})$. In words, these walks move forwards and backwards along the edge root for the first two steps, then never use the edge root again, thus they stay in one of the two branches. From Lemma~\ref{Lem:closed_wa_ver}, we now get
\begin{align}
C^*_{\ro(T)}(k;T)
&=C_{\ro(T_1)}(k-2;T_1)+C_{\ro(T_2)}(k-2;T_2)\nonumber\\
&\leq C_{\ro(G_1)}(k-2;G_1)+C_{\ro(G_2)}(k-2;G_2)\nonumber\\
\label{Eq:C_P}
&=C^*_{\ro(G)}(k;G).
\end{align}
We are left with walks that use the edge root, return immediately, and use the edge root again at some stage. The set of these walks is divided further, depending on the first time that the edge root is used again. For any $j\geq 1$, we consider the subsets $\mathcal{C}'^j_{\ro(T)}(k;T)$ and $\mathcal{C}'^j_{\ro(G)}(k;G)$
of $\mathcal{C}_{\ro(T)}(k;T)$ and $\mathcal{C}_{\ro(G)}(k;G)$ whose elements are the closed
walks with level sequence $(1,1,$ $i_1, i_2,\dots,i_{k-1})$, where $i_1=i_j=i_{j+1}=1$ and $(1,1)\neq (i_h,i_{h+1})$
for any $h\in \{1,\dots,j-1\}$. Such a walk can be split uniquely into a walk of length $j+1$ in 
$\mathcal{C}^*_{\ro(T)}(j+1;T)$ ($\mathcal{C}^*_{\ro(G)}(j+1;G)$, respectively) and a closed walk of length $k-j-1$ starting with the edge root. From \eqref{Eq:C_P} and Lemma~\ref{Lem:Adj_uv}, we obtain
\begin{align*}
C'^j_{\ro(T)}(k;T)
&= C^*_{\ro(T)}(j+1;T) \cdot \frac{1}{2}C_{\ro(T)}(k-j-1;T)\\
&\leq C^*_{\ro(G)}(j+1;G) \cdot \frac{1}{2}C_{\ro(G)}(k-j-1;G)\\
&=C'^j_{\ro(G)}(k;G)
\end{align*}
for any $j\geq 1$. We see that the greedy tree $G$ has more or at least equally many walks of each type as $T$, which completes the proof.
\end{proof}

\begin{lem}
\label{Lem:Closed_Pass}
Let $D$ be a leveled degree sequence of an edge-rooted tree and $G = G(D)$ the associated edge-rooted greedy tree. For any element 
$T\in \mathbb{T}_D$ and for any nonnegative integer $k$ we have
$$
C^{\ro(T)}(k;T)\leq C^{\ro(G)}(k;G).
$$
\end{lem}
\begin{proof}
Any element, say $W$, in $\mathcal{C}^{\ro(T)}(k;T)$ or $\mathcal{C}^{\ro(G)}(k;G)$ has a 
unique decomposition as
$
W=W_1W_2W_3
$
for some $W_1,W_2$ and $W_3$ satisfying the following conditions:
\begin{itemize}
 \item[$i)$] $W_2$ is a closed walk starting from the edge root, chosen to have maximal length.

 \item[$ii)$] $W_1$ and $W_3$ do not use the edge root but can possibly have length zero. By merging the end of $W_1$ with the beginning of $W_3$, we obtain a closed walk $W'$.
\end{itemize}
Under the conditions $i)$ and $ii)$, $W_3$ visits an end of the edge root only once (at its starting point), 
otherwise we could extend $W_2$. This means that $W$ can be uniquely recovered from $W'$ and $W_2$ by inserting $W_2$ into $W'$ at the last appearance of an end vertex of the edge root. So the number of possible walks $W$ is the number of possible walks $W'$ times the number of possible walks $W_2$.

Let $T_1$ and $T_2$ be the components 
of $T - \ro(T)$, and $G_1$ and $G_2$ those of $G-\ro(G)$. By Lemma~\ref{Lem:closed_wa_ver}, we know that 
\begin{align*}
C^{\ro(T_1)}(l;T_1)+C^{\ro(T_2)}(l;T_2)\leq C^{\ro(G_1)}(l;G_1)+C^{\ro(G_2)}(l;G_2)
\end{align*} 
for any nonnegative integer $l$. Hence, using Lemma~\ref{Lem:Start_Root} we have 
\begin{align*}
C^{\ro(T)}(k;T)
&= \sum_{k_1+k_2=k}\big(C^{\ro(T_1)}(k_1;T_1)C_{\ro(T_1),\ro(T_2)}(k_2;T)+C^{\ro(T_2)}(k_1;T_2)C_{\ro(T_2)\ro(T_1)}(k_2;T)\big)\\
&= \sum_{k_1+k_2=k}\big(C^{\ro(T_1)}(k_1;T_1)+C^{\ro(T_2)}(k_1;T_2)\big) \cdot \frac{1}{2}C_{\ro(T)}(k_2;T)\\
&\leq \sum_{k_1+k_2=k}\big(C^{\ro(G_1)}(k_1;G_1)+C^{\ro(G_2)}(k_1;G_2)\big) \cdot \frac{1}{2}C_{\ro(G)}(k_2;G)\\
&=C^{\ro(G)}(k;G).
\end{align*}
\end{proof}
The next theorem combines Theorem~\ref{Thm:Main_ver_root} and Lemma~\ref{Lem:Closed_Pass}.
\begin{thm}
\label{Th:Main_closed_wa_edge}
Let $D$ be a leveled degree sequence of an edge-rooted tree. For any nonnegative integer $k$ and all $T\in \mathbb{T}_D$, we have
$$
\M_k(T)=C(k;T)\leq C(k;G(D))=\M_k(G(D)).
$$
For sufficiently large even $k$, the inequality is strict unless $T$ and $G(D)$ are isomorphic.
\end{thm}
\begin{proof}
Use Theorem~\ref{Thm:Main_ver_root} to compare the number of closed walks of length $k$ not using 
the edge root, and Lemma~\ref{Lem:Closed_Pass} for those which pass through the edge root. The fact that the inequality in Theorem~\ref{Thm:Main_ver_root} is strict for sufficiently large $k$ by Lemma~\ref{Lem:Strict} implies that this is also the case here.
\end{proof}

\subsection{Main result}
The main result of this section is the fact that if we fix a degree sequence $D$, then among 
all trees with degree sequence $D$, the greedy tree $G(D)$ has the maximum number of closed walks of any given length. $G(D)$ is not always the unique element of $\mathbb{T}_D$ which reaches the maximum number of fixed 
length closed walks: for instance, for any $T\in \mathbb{T}_D$, we have $\mathcal{C}(2;T)=2|E(T)|$, which only depends on $D$.
\begin{thm}
\label{Th:Main}
Let $D$ be a degree sequence of a tree. For any element 
$T\in \mathbb{T}_D$ and any $k \geq 0$, we have
$$\M_k(T)=C(k;T)\leq C(k;G(D))=\M_k(G(D)).$$
Moreover, the inequality is strict for sufficiently large even $k$ if $T$ and $G(D)$ are not isomorphic.
\end{thm}
\begin{proof}
If it is possible to choose an edge or a vertex as root such that $T$ is not level greedy, then we let 
$T_1$ be the level greedy tree with the same leveled degree sequence as $T$. We iterate this process: if an edge or vertex root can be chosen such that $T_l$ is not level greedy, replace it by the corresponding level greedy tree, which we denote by $T_{l+1}$. Then $\M_k(T_{l+1}) \geq \M_k(T_l)$ for all $k \geq 0$, and for sufficiently large even $k$, the inequality is strict. Therefore, no infinite loops are possible in this process.

Hence there exists an integer $m$ 
such that $T_m$ is level greedy with respect to any choice of vertex or edge root. This tree $T_m$ satisfies the ``semi-regular'' property defined in \cite{sze11}, and hence it is a greedy tree. From Theorems~\ref{Thm:Main_ver_root} and \ref{Th:Main_closed_wa_edge}, we obtain
$$
C(k;T)\leq C(k;T_1)\leq \dots\leq C(k;T_m) = C(k;G(D))
$$
for any $k\geq 0$, with strict inequality for sufficiently large even $k$. 
\end{proof}

\begin{rem} While the inequality in Theorem~\ref{Th:Main} is strict for sufficiently large $k$, there is no ``universal'' $k$ with this property: for every $k$, there exists some degree sequence $D$ and a tree $T$ with degree sequence $D$ that is not isomorphic to the greedy tree $G = G(D)$ such that
$$\M_{\ell}(T) = \M_{\ell}(G),\qquad \ell=0,1,\ldots,k.$$
Consider, for instance, the degree sequence $D=(3,3,2,2,\ldots,2,1,1,1,1)$, where the number of $2$s is $4r-2$ for some integer $r \geq 1$. The greedy tree $G = G(D)$ consists of two neighboring vertices of degree $3$ to which paths are attached: two paths of length $r$ to one of the two, two paths of length $r+1$ to the other. Now let $T$ be the tree where one of the paths of length $r$ in $G$ is interchanged with one of the paths of length $r+1$.

$T$ and $G$ have the same number of (closed) walks of any length that do not contain the vertices of degree $3$, since the forests resulting when the two are removed are isomorphic. Moreover, the subtrees of $T$ and $G$ consisting of vertices whose distance from the degree $3$ vertices is at most $r$ are isomorphic as well. Thus
$$\M_{\ell}(T) = \M_{\ell}(G),\qquad \ell \leq 2r.$$
\end{rem}

\subsection{Consequences of the main result}

Several corollaries follow immediately from our main theorem. In particular, in view of \eqref{Eq:ef}, we obtain the following corollary:
\begin{cor}
For any function $f(x) = \sum_{k=0}^{\infty} a_kx^k$ with nonnegative coefficients and for any tree $T$ with degree sequence $D$, we have
$$\E_f(T)\leq \E_f(G(D)),$$
where $E_f$ is defined as in~\eqref{Eq:efdefi}. If the even part of $f$ is not a polynomial (i.e., $a_k > 0$ for infinitely many even values of $k$), then the inequality is strict unless $T$ is isomorphic to $G(D)$. In particular,
$$\EE(T) < \EE(G(D))$$
for all $T\in \mathbb{T}_D$ that are not isomorphic to $G(D)$.
\end{cor}

Moreover, we also obtain one of the main results of \cite{biyikoglu2008graphs} as another corollary, since the spectral radius $\rho(T)$ of a tree $T$ is equal to the limit $\lim_{\ell \to \infty} \sqrt[2\ell]{\M_{2\ell}(T)}$.

\begin{cor}\label{cor:specrad}
Among all trees with degree sequence $D$, the greedy tree $G(D)$ has the largest spectral radius $\rho(G(D))$.
\end{cor}

In \cite{biyikoglu2008graphs}, it was also shown that the greedy tree is unique with this property.

The Estrada index is just one of in principle infinitely many graph invariants of the form $E_f$. One could certainly conceive of a ``Hyper-Estrada index'', for example:
$$\operatorname{EEE}(G) = \sum_{i=1}^n e^{e^{\lambda_i}}.$$
A somewhat more natural example is the following: note that the characteristic polynomial of a graph $G$ is given by
$$P_G(x) = \prod_{i=1}^n (x-\lambda_i) = x^n \prod_{i=1}^n \left( 1 - \frac{\lambda_i}{x} \right).$$
If $x$ is greater than the spectral radius, then we can take the logarithm and expand it into a power series:
\begin{align*}
\log P_G(x) &= n \log x + \sum_{i=1}^n \log \left( 1 - \frac{\lambda_i}{x} \right) = n \log x - \sum_{i=1}^n \sum_{k=1}^{\infty} \frac{1}{k} \cdot \frac{\lambda_i^k}{x^k} \\
&= n \log x - \sum_{k=1}^{\infty} \frac{M_k(G)}{k x^k}.
\end{align*}
This formula, together with our main result, implies the following statement:

\begin{cor}\label{cor:charpoly}
For any tree $T$ with degree sequence $D$ and any $x > \rho(G(D))$, the inequality
$$P_T(x) \geq P_{G(D)}(x)$$
holds, with equality only if $T$ is isomorphic to $G(D)$.
\end{cor}

\section{Trees with different degree sequences}
\label{Sec:Diff}

In this section, we compare greedy trees with different degree sequence, in a similar way as it was done in \cite{biyikoglu2008graphs,Andriantiana2012,andriantiana2013greedy,zhang12}. This allows us to determine the maximal spectral moments of trees with different restrictions, e.g. given maximum degree or number of leaves.

To this end, we use a transformation on level greedy trees, where branches are moved between vertices at the same level. We study the 
effect of such a transformation on the number of closed walks of given length. Unlike the procedure in the proof of Theorem~\ref{Th:Main}, the transformation that we consider in the following lemma does not preserve the degree sequence.

For any vertex $v$ in a rooted tree $T$, we denote by $T_v$ the rooted tree spanned by 
$v$ and all its descendants, where $v$ is chosen to be the root.
\begin{lem}
\label{Lem:Snake}
Let $D=((i_{1,1}),(i_{2,1},\dots,i_{2,k_2}),\dots,(i_{n,1},\dots,i_{n,k_n}))$ be a leveled 
degree sequence of a (vertex) rooted tree. For some $i$ and $j$ with $1< i < L(D)$ and $1<j\leq k_i$, let $B$ be a branch of $g^i_j$ in the level greedy tree $G = G(D)$ which does not contain the root. Choose the neighbor of $g_j^i$ in $B$ to be the 
root of $B$. Let $T=G-g^i_j\ro(B)+g^i_{j'}\ro(B)$ for some $j'<j$  (see Figure~\ref{Fig:moving}). Then we have 
$$C(k;T)\geq C(k;G)$$ for any nonnegative integer $k$. For even $k \geq 4$, the inequality is strict.
\end{lem}

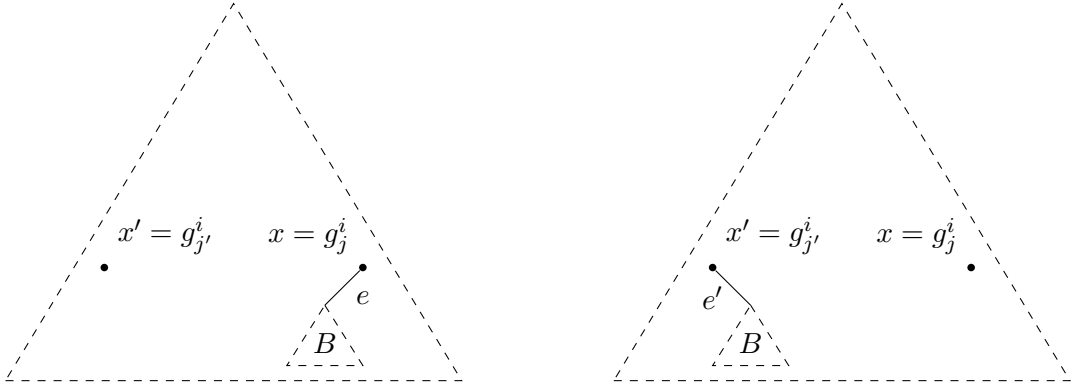
\begin{figure}[htbp]
\centering
    \begin{tikzpicture}[scale=1]
        \node[fill=black,label=above right:{$x' = g^i_{j'}$},circle,inner sep=1pt] (v11) at (1.3,1.5) {};
        \node[fill=black,label=above left:{$x = g^i_j$},circle,inner sep=1pt] (v21) at (4.7,1.5) {};
        \draw[dashed] (0,0)--(6,0)--(3,5)--(0,0);
        \draw (v21) -- (4.2,1);
        \draw[dashed] (4.2,1)--(4.7,0.2)--(3.7,0.2)--(4.2,1);
        \node at (4.2,0.5) {$B$};
        \node at (4.7,1.1) {$e$};

       \node[fill=black,label=above right:{$x' = g^i_{j'}$},circle,inner sep=1pt] (v12) at (9.3,1.5) {};
        \node[fill=black,label=above left:{$x = g^i_j$},circle,inner sep=1pt] (v22) at (12.7,1.5) {};
        \draw[dashed] (8,0)--(14,0)--(11,5)--(8,0);
        \draw (v12) -- (9.8,1);

        \draw[dashed] (9.8,1)--(10.3,0.2)--(9.3,0.2)--(9.8,1);
        \node at (9.8,0.5) {$B$};
        \node at (9.3,1.1) {$e'$};

    \end{tikzpicture}
\caption{Moving a branch: the level greedy tree $G$ (left) and the resulting tree $T$ (right).}
\label{Fig:moving}
\end{figure}

\begin{proof}
We use the same labels for vertices in $T$ as in $G$. For notational convenience, set 
$x= g^i_j$, $x' = g^i_{j'}$, $e=g^i_j\ro(B)$ and $e'=g^i_{j'}\ro(B).$

It is clear that $C(k;T)-C^{e'}(k;T) = C(k;G)-C^{e}(k;G)$
because $T-e'=G-e$. Thus it suffices to prove
\begin{equation}\label{eq:through_e}
C^{e'}(k;T) \geq C^{e}(k;G).
\end{equation}
Let $v=g^{i'}_l$ be the closest common
ancestor of $x = g^i_j$ and $x' = g^i_{j'}$ in $G$, and let
$u = g^{i'+1}_{h}$ and $u' = g^{i'+1}_{h'}$ be the neighbors of $v$ in the branches containing $g^i_{j}$ and $g^i_{j'}$ respectively. 
\begin{figure}[htbp]
\centering
    \begin{tikzpicture}[scale=1]
        \node[fill=black,label=right:{$v$},circle,inner sep=1pt] (u) at (0,1) {};

        \node[fill=black,label=right:{$u'$},circle,inner sep=1pt] (v1) at (-1,0) {};
        \draw[dashed] (v1)--(-1,-4)--(-0.25,-4)--(v1);
        \node[fill=black,circle,inner sep=1pt] (v11) at (-1.5,-1.5) {};
        \draw[dashed] (v11)--(-1.85,-4)--(-1.1,-4)--(v11);
        \node[fill=black,label=right:{$x'$},circle,inner sep=1pt] (v12) at (-2.75,-3) {};
        \node[fill=white,label=right:{$C'_{i-i'}$},circle,inner sep=0pt] () at (-3.35,-3.75) {};
        \node[fill=white,label=right:{$C'_2$},circle,inner sep=0pt] () at (-1.95,-3.75) {};
        \node[fill=white,label=right:{$C'_1$},circle,inner sep=0pt] () at (-1.15,-3.75) {};
        \draw[dashed] (v12)--(-3.5,-4)--(-2,-4)--(v12);
        \node[fill=black,circle,inner sep=0.5pt] () at (-2.55,-2.7) {};
        \node[fill=black,circle,inner sep=0.5pt] () at (-2.17,-2.22) {};
        \node[fill=black,circle,inner sep=0.5pt] () at (-1.8,-1.85) {};

        \node[fill=black,label=right:{$u$},circle,inner sep=1pt] (v2) at (1,0) {};
        \draw[dashed] (v2)--(1,-4)--(0.25,-4)--(v2);
        \node[fill=black,circle,inner sep=1pt] (v21) at (1.5,-1.5) {};
        \draw[dashed] (v21)--(1.85,-4)--(1.1,-4)--(v21);
        \node[fill=black,label=left:{$x$},circle,inner sep=1pt] (v22) at (2.75,-3) {};
        \node[fill=white,label=left:{$C_{i-i'}$},circle,inner sep=0pt] () at (3.35,-3.75) {};
        \node[fill=white,label=left:{$C_2$},circle,inner sep=0pt] () at (1.95,-3.75) {};
        \node[fill=white,label=left:{$C_1$},circle,inner sep=0pt] () at (1.15,-3.75) {};
        \node[fill=white,label=above:{\large $G(D)$},circle,inner sep=0pt] () at (5,-3.75) {};
        \draw[dashed] (v22)--(3.5,-4)--(2,-4)--(v22);
        \node[fill=black,circle,inner sep=0.5pt] () at (2.55,-2.7) {};
        \node[fill=black,circle,inner sep=0.5pt] () at (2.17,-2.22) {};
        \node[fill=black,circle,inner sep=0.5pt] () at (1.8,-1.85) {};

        \draw (v1)--(u)--(v2);
        \draw (v21)--(v2);
        \draw (v1)--(v11);
        \draw (u)--(0.5,1.5);
        \draw[dashed] (1,2.5)--(-5,-4.1)--(7,-4.1)--(1,2.5);
    \end{tikzpicture}
\caption{Decomposition of $G_v$ in the proof of Lemma~\ref{Lem:Snake}}
\label{Fig:isom_VertRoot}
\end{figure}
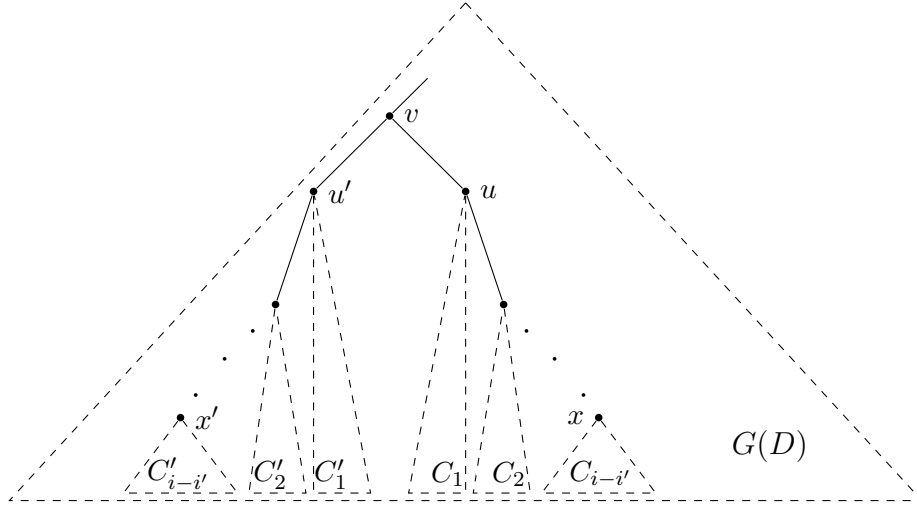

Since $G$ is level greedy, if we decompose $G_{v}$ as in Figure~\ref{Fig:isom_VertRoot}, 
then there is an isomorphism preserving roots between $C_r$ and a subgraph of $C'_r$ for any $r\in\{1,2,\dots,i-i'\}$.
Therefore one can find an injective homomorphism, say 
$f:V\big(G_{u}\big)\longrightarrow V\big(T_{u'}\big)$, which satisfies
$f(u)=u'$, $f(x)=x'$ and $f(e)=e'$. 

The map 
\begin{align*}
F:\mathcal{C}^{e}(k;G)-\mathcal{C}^{v,e}(k;G) &\longrightarrow \mathcal{C}^{e'}(k;T)-\mathcal{C}^{v,e'}(k;T)\\
 w_1\dots w_{k+1} &\longmapsto f(w_1)\dots f(w_{k+1})
\end{align*}
is injective because $f$ is injective. We also define a map 
$$F':\mathcal{C}^{v}(k;G) \longrightarrow \mathcal{C}^{v}(k;T)$$
in a recursive way. Let $W=w_1\dots w_{k+1}\in \mathcal{C}^{v}(k;G),$ 
and let $m$ and $M$ be, respectively, the smallest and largest integers such that 
$w_m=w_M=v$ and $1<m\leq M<k+1$, if there exist such integers. Then we define:
\begin{itemize}
\item  If $v\notin \{w_2,\dots,w_k\}$ (and hence $w_1=w_{k+1}=v$)
and $w_sw_{s+1}\neq e$ for any  $s=1,\dots, k,$ then $F'(W)=w_1\dots w_{k+1}.$
\item If $v\notin \{w_2,\dots,w_k\}$ and $w_sw_{s+1}= e$ for some  $s\in\{1,\dots, k\}$, then
$$F'(W)=w_1f(w_2)\dots f(w_k)w_{k+1}.$$
\item Otherwise we set
$F'(W)=\phi(w_1\dots w_{m-1})F'(w_m\dots w_M)\phi(w_{M+1}\dots w_{k+1}),$
where $\phi(w_1\dots w_{m-1})=f(w_1)\dots f(w_{m-1})$ if $w_sw_{s+1}= e$ for some $s\in\{1,\dots, m-2\}$,
and $\phi(w_1\dots w_{m-1})=w_1\dots w_{m-1}$ otherwise. 
\end{itemize}
In words, we break a walk into pieces separated by visits to vertex $v$. Each piece is either kept the same (if it does not contain $e$) or replaced by its image under the injection $f$ if it contains $e$. Since the decomposition is unique and $f$ is injective, the so constructed map $F'$ is also an injection, and so is its restriction to $\mathcal{C}^{v,e}(k;G)$. This proves inequality~\eqref{eq:through_e} and thus the main inequality.

For even $k \geq 4$, the inequality is strict, since $F$ is not surjective. The degree of $x$ in 
$G$ is strictly less than the degree of $x'$ in $T$ by construction. Hence, there is an edge $e''$ 
incident to $x'$ that does not have a preimage under $F$, and so is any walk starting from $e'$ and uses $e''$.
There is such a closed walk for arbitrary even length larger than $4$.
\end{proof}

\begin{lem}
\label{Lem:Snake_e}
Let $D=((i_{1,1},i_{1,2}),(i_{2,1},\dots,i_{2,k_2}),\dots,(i_{n,1},\dots,i_{n,k_n}))$ be a leveled 
degree sequence of an edge-rooted tree. For some $i$ and $j$ with $1\leq i < L(D)$ and $1<j\leq k_i$, let $B$ be a 
branch of $g^i_j$ in the level greedy tree $G = G(D)$ which does not contain the root. Take the neighbor of $g^i_j$ in $B$ as root of $B$.
Let $T=G-g^i_j\ro(B)+g^i_{j'}\ro(B)$ for some $j'<j$. Then we have
$$
C(k;T)\geq C(k;G).
$$
for any nonnegative integer $k$. For even $k \geq 4$, the inequality is strict.
\end{lem}
\begin{proof}
Again, we keep the labels of vertices of $G$ in $T$. For simplicity, we write $x = g^i_j$, $x' = g^i_{j'}$, $e=g^i_j\ro(B),$ $e'=g^i_{j'}\ro(B)$, $r=\ro(G)$ and $r'=\ro(T).$

Let $G_1$ and $G_2$ be the components of $G-r$, and $T_1$ and $T_2$ those of $T-r'$, such that $|V(G_1)|\geq |V(G_2)|$ and $|V(T_1)|\geq |V(T_2)|$.

If $x$ and $x'$ are both vertices of the same component $G_m$, then the proof is exactly the same as that of Lemma~\ref{Lem:Snake}. So from now on, we assume that $x\in V(G_2)$ and $x'\in V(G_1)$.

If $G$ is decomposed as in Figure~\ref{Fig:isom}, then $C_r$ has a copy preserving levels 
in $C'_r$ for any $1\leq r \leq i$. Because of this fact, we know that one can find a level 
preserving injective homomorphism, say $f$, between $G_2$ and 
$T_1$ (which has $G_1$ as a subgraph) which satisfies $f(g^1_2)=g^1_1$, $f(x)=x'$ and $f(e)=e'$.
\begin{figure}[htbp]
\centering
    \begin{tikzpicture}[scale=1]
        \node[fill=black,label=above:{$g^1_1$},circle,inner sep=1pt] (v1) at (-1,0) {};
        \draw[dashed] (v1)--(-1,-4)--(-0.25,-4)--(v1);
        \node[fill=black,circle,inner sep=1pt] (v11) at (-1.5,-1.5) {};
        \draw[dashed] (v11)--(-1.85,-4)--(-1.1,-4)--(v11);
        \node[fill=black,label=right:{$x'$},circle,inner sep=1pt] (v12) at (-2.75,-3) {};
        \node[fill=white,label=right:{$C'_i$},circle,inner sep=0pt] () at (-3.1,-3.75) {};
        \node[fill=white,label=right:{$C'_2$},circle,inner sep=0pt] () at (-1.95,-3.75) {};
        \node[fill=white,label=right:{$C'_1$},circle,inner sep=0pt] () at (-1.15,-3.75) {};
        \draw[dashed] (v12)--(-3.5,-4)--(-2,-4)--(v12);
        \node[fill=black,circle,inner sep=0.5pt] () at (-2.55,-2.7) {};
        \node[fill=black,circle,inner sep=0.5pt] () at (-2.17,-2.22) {};
        \node[fill=black,circle,inner sep=0.5pt] () at (-1.8,-1.85) {};

        \node[fill=black,label=above:{$g^1_2$},circle,inner sep=1pt] (v2) at (1,0) {};
        \draw[dashed] (v2)--(1,-4)--(0.25,-4)--(v2);
        \node[fill=black,circle,inner sep=1pt] (v21) at (1.5,-1.5) {};
        \draw[dashed] (v21)--(1.85,-4)--(1.1,-4)--(v21);
        \node[fill=black,label=left:{$x$},circle,inner sep=1pt] (v22) at (2.75,-3) {};
        \node[fill=white,label=left:{$C_i$},circle,inner sep=0pt] () at (3.1,-3.75) {};
        \node[fill=white,label=left:{$C_2$},circle,inner sep=0pt] () at (1.95,-3.75) {};
        \node[fill=white,label=left:{$C_1$},circle,inner sep=0pt] () at (1.15,-3.75) {};
        \draw[dashed] (v22)--(3.5,-4)--(2,-4)--(v22);
        \node[fill=black,circle,inner sep=0.5pt] () at (2.55,-2.7) {};
        \node[fill=black,circle,inner sep=0.5pt] () at (2.17,-2.22) {};
        \node[fill=black,circle,inner sep=0.5pt] () at (1.8,-1.85) {};

        \draw (v21)--(v2)--(v1)--(v11);
    \end{tikzpicture}
\caption{Decomposition of $G$ in the proof of Lemma~\ref{Lem:Snake_e}}
\label{Fig:isom}
\end{figure}
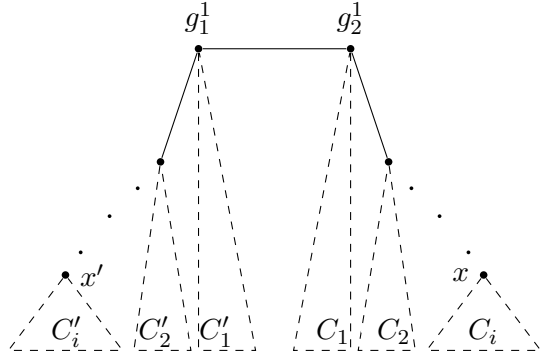

Since we deal with closed walks, we are only interested in even $k=2l.$
We know that 
$$
C(2l;T) - C^{r'}(2l;T)-[C(2l;G) - C^{r}(2l;G)]=C^{e'}(2l;T_1) - C^{e}(2l;G_2)
$$
is nonnegative: as $f$ is injective, so is the map 
\begin{align*}
F:\mathcal{C}^{e}(2l;G_2)&\longrightarrow \mathcal{C}^{e'}(2l;T_1)\\
 w_1\dots w_{k+1}&\longmapsto f(w_1)\dots f(w_{k+1}).
\end{align*}
Since $f$ is level preserving, we can even choose an arbitrary level sequence $S$ of walks 
without two consecutive $1$s and still have
\begin{align}
C^{e'}(S;T_1)&\geq C^{e}(S;G_2), \nonumber \\
C(S,(T_1-B)\cup T_2)&= C(S;G_1\cup (G_2-B)), \label{eq:ls-ineq}
\end{align}
and hence $C(S;T_1\cup T_2)\geq C(S;G_1\cup G_2).$
Now we are left to show that 
\begin{align}
\label{Eq:Sec}
C^{r'}(2l;T)-C^{r}(2l;G)=C^{e',r'}(2l;T)-C^{e,r}(2l;G)\geq 0
\end{align}
for any integer $l\geq 1$. Before that let us first show that 
\begin{equation}
\label{Eq:Inter}
C^{e'}_{r'}(2l;T)\geq C^{e}_{r}(2l;G)
\end{equation}
for any positive integer $l$. Note the subtle difference between $C^{e,r}(2l;G)$ and $C^{e}_{r}(2l;G)$ here: the former counts walks that pass through $r$ \emph{at some stage}, while the latter counts walks that start with $r$. We reason by induction on $l$. For $l=1$  we have
$
C^{e'}_{r'}(2;T)=C^{e}_{r}(2;G)=0.
$
Assume that \eqref{Eq:Inter} holds whenever $l\leq m$ for some $m\geq 1.$ Since 
$
C_{r'}(2l;T)-C^{e'}_{r'}(2l;T)=C_{r}(2l;G)-C^{e}_{r}(2l;G)
$
for all $l$, the induction hypothesis also implies that 
$
C_{r'}(2l;T) \geq C_{r}(2l;G)
$
for all $l\leq m.$

Consider now the case where $l=m+1$. Let $\mathcal{C}^{e'}_{r'}(2l;T)=P_1(l)\cup Q_1(l)\cup R_1(l)$ and 
$\mathcal{C}^{e}_{r}(2l;G)=P_2(l)\cup Q_2(l)\cup R_2(l)$, where the $P_i(l)$'s contain 
walks whose level sequences start with $1,1,1,1$, the $Q_i(l)$'s contain walks whose level sequences start with $1,1,1,2$, and the level sequences of the elements of the $R_i(l)$'s 
start with $1,1,2$. The induction hypothesis implies
\begin{align*}
|P_1(l)|
&=C^{e'}_{r'}(2(l-1);T) \geq C^{e}_{r}(2(l-1);G)=|P_2(l)|.
\end{align*}
It is easy to check that $|Q_1(1)|=|Q_2(1)|=0,$ $|Q_1(2)|=|Q_2(2)|\in \{0,1\}$. For $l\geq 3$, we define  for any $j$ with $2 \leq j \leq 2l-3$ the 
subset $Q_i^j(l)$ of $Q_i(l)$ whose elements have level sequence $(1,1,1,i_1,\dots,i_{2l-2})$, 
where $i_j=i_{j+1}=1$ and $(i_s,i_{s+1})\neq(1,1)$ for $s = 1,\ldots,j-1$. Finally, 
$$Q_i^{2l-2}(l)=Q_i(l)-\bigcup_{j=2}^{2l-3} Q_i^j(l)$$
is the subset of $Q_i(l)$ whose elements 
have level sequence $(1,1,1,i_1,\dots,i_{2l-2})$, where $i_1=2$ and $(1,1)\neq (i_s,i_{s+1})$ for 
$s=1,\dots,2l-3.$ Set 
\begin{align*}
\mathbb{S}^1_j=\{(1,i_1,\dots,i_{j-1},1): i_1 = i_{j-1} = 2, (1,1)\neq(i_s,i_{s+1})\text{ for } s\in\{1,\dots,j-2\} \}.
\end{align*}
Now we decompose walks in $Q_1^j$ and $Q_2^j$: any such walk consists of two steps along the edge root (forwards and backwards), then continues to higher levels and returns to the first level after $j$ steps (possibly earlier as well, but without ever using the edge root). We call this part $U_1$; its level sequence lies in $\mathbb{S}^1_j$. Thereafter, the walk continues for another $2l-j-2$ steps, starting with the edge root; this part is called $U_2$. Since we know that a walk in  $Q_1^j$ has to pass through $e'$, we have the following possibilities:
\begin{itemize}
\item The walk $U_1$ uses $e'$ (which means that it lies entirely in $T_1$), the walk $U_2$ is arbitrary.
\item The walk $U_1$ does not use $e'$, but stays in $T_1$ (thus it lies in $T_1 - B$), the walk $U_2$ contains $e'$. 
\item The walk $U_1$ lies in $T_2$, thus it does not use $e'$. Then the walk $U_2$ has to contain $e'$.
\end{itemize}
For $Q_2^j$, there are three analogous possibilities. Making use of this decomposition, Lemma~\ref{Lem:Adj_uv},~\eqref{eq:ls-ineq} and the induction hypothesis, we obtain
\begin{align*}
&|Q_1^j(l)|\\
&=\sum_{S\in\mathbb{S}^1_j} C^{e'}(S;T_1) C_{\ro(T_1),\ro(T_2)}(2l-j-2;T)+\sum_{S\in\mathbb{S}^1_j}C(S;T_1-B) C^{e'}_{\ro(T_1),\ro(T_2)}(2l-j-2;T)\\
&\qquad+\sum_{S\in\mathbb{S}^1_j} C(S;T_2) C^{e'}_{\ro(T_2),\ro(T_1)}(2l-j-2;T)\\
&=\sum_{S\in\mathbb{S}^1_j} C^{e'}(S;T_1) \cdot \frac{1}{2} C_{r'}(2l-j-2;T)+\sum_{S\in\mathbb{S}^1_j}\left(C(S;T_1-B)+C(S;T_2)\right) \cdot \frac{1}{2} C^{e'}_{r'}(2l-j-2;T)\\
&\geq\sum_{S\in\mathbb{S}^1_j} C^{e}(S;G_2) \cdot \frac{1}{2} C_{r}(2l-j-2;G)+\sum_{S\in\mathbb{S}^1_j}\left(C(S;G_1)+C(S;G_2-B)\right) \cdot \frac{1}{2} C^{e}_{r}(2l-j-2;G)\\
&=|Q_2^j(l)|
\end{align*}
for all $j$ such that $2 \leq j \leq 2l-3$. For $j = 2l-2$, walk $U_2$ is empty, so we have
$$
|Q_1^{2l-2}(l)| =\sum_{S\in\mathbb{S}^1_j} C^{e'}(S;T_1) \geq \sum_{S\in\mathbb{S}^1_j} C^{e}(S;G_2) = |Q_2^{2l-2}(l)|.
$$
We conclude with the third subclass of walks whose level sequences start with $1,1,2$.
For any $j$ with $2 \leq j \leq 2l-2$, let $R_i^j(l)$ be the subset of $R_i(l)$ whose elements have level sequence $(1,1,i_1,\dots,i_{2l-1})$, 
where $i_j=1$ and $1\notin\{i_1,\dots,i_{j-1}\}$. The case that $j=2l-1$ is not interesting since it does 
not correspond to any closed walk. We decompose $R_1^j(l)$ and $R_2^j(l)$ in a similar way as we decomposed $Q_1^j(l)$ and $Q_2^j(l)$. Define 
$$\mathbb{S}^2_j=\{(1,i_1,\dots,i_{j-1},1): 1\notin \{i_1,\dots,i_{j-1}\} \}.$$
A walk in $R_1^j(l)$ (or $R_2^j(l)$) consists of a step along the edge root, then moves to higher levels and only returns to the first level after $j$ steps. This part of $j$ steps has a level sequence in $\mathbb{S}^2_j$, the rest forms a closed walk starting with the edge root. Dividing into three cases again, depending on which part contains $e'$ ($e$, respectively), we obtain
\begin{align*}
|R_1^j(l)|
&=\sum_{S\in\mathbb{S}^2_j} C^{e'}(S;T_1) C_{\ro(T_2),\ro(T_1)}(2l-j;T)+\sum_{S\in\mathbb{S}^2_j} C(S;T_1-B) C^{e'}_{\ro(T_2),\ro(T_1)}(2l-j;T)\\
&\qquad+\sum_{S\in\mathbb{S}^2_j} C(S;T_2) C^{e'}_{\ro(T_1),\ro(T_2)}(2l-j;T)\\
&=\sum_{S\in\mathbb{S}^2_j} C^{e'}(S;T_1) \cdot \frac{1}{2} C_{r'}(2l-j;T)+\sum_{S\in\mathbb{S}^2_j}\left(C(S;T_1-B)+C(S;T_2)\right)\cdot \frac{1}{2}C^{e'}_{r'}(2l-j;T)\\
&\geq\sum_{S\in\mathbb{S}^2_j}C^{e}(S;G_2) \cdot \frac{1}{2} C_{r}(2l-j;G)+\sum_{S\in\mathbb{S}^2_j}\left(C(S;G_1)+C(S;G_2-B)\right) \cdot \frac{1}{2} C^{e}_{r}(2l-j;G)\\
&=|R_2^j(l)|.
\end{align*}
This completes the proof of \eqref{Eq:Inter}. We now proceed to the proof of \eqref{Eq:Sec}, making use of a similar argument as in Lemma~\ref{Lem:Closed_Pass}.
Any element, say $W$, in $\mathcal{C}^{e',r'}(2l;T)$ or $\mathcal{C}^{e,r}(2l;G)$
has a unique decomposition 
\begin{equation}
\label{Eq:UniqDec}
W=W_1W_2W_3,
\end{equation}
where $W_2$ is a closed walk that starts with the edge root and is chosen to have maximal length and 
$W'=W_1W_3$ forms a closed walk which never uses the edge root, but passes at least once through one of its ends (unless it is empty). The decomposition \eqref{Eq:UniqDec} is unique (as it was explained in the proof of Lemma~\ref{Lem:Closed_Pass}). Now let $$\mathbb{S}^3_j=\{(i_1,\dots,i_{j+1}): i_s=1\text{ for some }1\leq s \leq j+1\}.$$ 
The walk $W'$ has a level sequence in $\mathbb{S}^3_j$ for some $j$. Again, there are three possibilities for a walk in $\mathcal{C}^{e',r'}(2l;T)$:
\begin{itemize}
\item The walk $W'$ contains $e'$, and thus lies entirely in $T_1$, and $W_2$ is arbitrary.
\item The walk $W'$ does not contain $e'$, but still lies in $T_1$ (thus entirely in $T_1 - B$), and $W_2$ uses $e'$.
\item The walk $W'$ lies in $T_2$, thus does not use $e'$. Then $W_2$ has to use $e'$.
\end{itemize}
There are three analogous possibilities for $\mathcal{C}^{e,r}(2l;G)$. We obtain
\begin{align*}
&C^{e',r'}(2l;T) \\
&=\sum_{j=0}^{2l-2}\sum_{S\in\mathbb{S}^3_j} C^{e'}(S;T_1) C_{\ro(T_1),\ro(T_2)}(2l-j;T)+C(S;T_1-B) C^{e'}_{\ro(T_1),\ro(T_2)}(2l-j;T)\\
&\qquad+C(S;T_2) C^{e'}_{\ro(T_2),\ro(T_1)}(2l-j;T)\\
&=\sum_{j=0}^{2l-2}\sum_{S\in\mathbb{S}^3_j} C^{e'}(S;T_1) \cdot \frac{1}{2} C_{r'}(2l-j;T)+\left(C(S;T_1-B)+ C(S;T_2)\right) \cdot \frac{1}{2} C^{e'}_{r'}(2l-j;T)\\
&\geq\sum_{j=0}^{2l-2}\sum_{S\in\mathbb{S}^3_j} C^{e}(S;G_2) \cdot \frac{1}{2} C_{r}(2l-j;G)+\left(C(S;G_1)+ C(S;G_2-B)\right) \cdot \frac{1}{2} C^{e}_{r}(2l-j;G)\\
&= C^{e,r'}(2l;G).
\end{align*}
This concludes the proof of~\eqref{Eq:Sec} and thus the theorem. As in the previous lemma, the inequality is strict for even $k \geq 4$ since the map $F$ is not surjective.

\end{proof}

Given two degree sequences $B\preccurlyeq D$ of trees, by iteratively transferring branches, we can transform 
$G(B)$ to become an element of $\mathbb{T}_D$. As seen in the proof of the next theorem, it turns out that 
it is always enough to only use transfers of the type described in the two Lemmas~\ref{Lem:Snake} 
and \ref{Lem:Snake_e} to obtain an element of $\mathbb{T}_D$ from $G(B)$, showing that $G(D)$ has more closed walks of any length than $G(B)$. This parallels analogous results for e.g. the number of subtrees \cite{zhang12} or the spectral radius \cite{biyikoglu2008graphs}.
\begin{thm}
 \label{Th:comp}
Let $D=(d_1,\dots,d_n)$ and $B=(b_1,\dots,b_n)$ be degree sequences of trees of the same order 
such that $B\preccurlyeq D$. Then for any integer $k\geq 0$ we have
$$
C(k;G(B)) \leq C(k;G(D)).
$$
If $B \neq D$ and $k$ is even and $\geq 4$, then the inequality is strict.
\end{thm}

\begin{proof}
The statement is obvious for $B = D$. From now, we assume that there exists some $i_0$ such that 
$b_{i_0}\neq d_{i_0}$. Since
\begin{equation}
\label{Eq:Mjoo}
\sum_{i=1}^nb_i=\sum_{i=1}^nd_i,
\end{equation}
we know that the set $\{i  : d_i\neq b_i\}$ must have at least two elements. Let 
$l=\min \{i  : d_i\neq b_i\}$ and $m=\max \{i  : d_i\neq b_i\}.$ We must have $b_{l}< d_{l}$,  
$b_{m}> d_m$ and hence $b_{l-1}=d_{l-1}\geq d_l\geq b_l+1$ and $b_{m+1}=d_{m+1}\leq d_m\leq b_m-1$.
Therefore, $B_1 = (b_1,\dots,b_{l-1},b_l+1,b_{l+1},\dots,b_{m-1},b_m-1,b_{m+1},\dots,b_n)$
is a valid degree sequence. It is easy to see that $B\preccurlyeq B_1.$ Consider two vertices $u$ 
and $v$ in the greedy tree $G(B)$ such that $\deg u=b_l$ and $\deg v=b_m$. 

\medskip

\noindent {\bf Case 1:} The length of the path in $G(B)$ joining $u$ and $v$ is even. Let $w$ be the middle vertex of this path. Consider $G(B)$ as a level greedy tree whose root is $w$. Then $u$ and $v$ are on the same level, say level $h$. We have $u = g_i^h$ and $v = g_j^h$ for some $i < j$.
Let $w = g_r^{h+1}$ be a child of $v = g_j^h$, 
and let $H = G(B)_{w}$ be the branch rooted at $w$.

Consider $T=G(B)-vw+uw$; the degree sequence of $T$ is $B_1$. By Theorem~\ref{Thm:Main_ver_root} 
and Lemma~\ref{Lem:Snake}, it follows that
$$C(k;G(B_1)) \geq C(k;T) \geq C(k;G(B))$$
for all $k \geq 0$. 

\medskip

\noindent {\bf Case 2:} The length of the path in $G(B)$ joining $u$ and $v$ is odd. The 
argument is analogous to the previous case: we choose the middle edge of the path as root and then 
we use Theorem~\ref{Th:Main_closed_wa_edge} and Lemma~\ref{Lem:Snake_e} instead of 
Theorem~\ref{Thm:Main_ver_root} and Lemma~\ref{Lem:Snake}.

\medskip

In either case, we have
$$C(k;G(B_1)) \geq C(k;G(B))$$
for all $k \geq 0$. We repeat this process to obtain a sequence of degree sequences $B_0 = B, B_1,B_2,\ldots,B_r = D$ such that
$
B = B_0 \preccurlyeq B_1 \preccurlyeq \dots \preccurlyeq B_r = D
$
and 
$$
C(k;G(B)) =  C(k;G(B_0)) \leq C(k;G(B_1))\leq \dots \leq C(k;G(B_r)) = C(k;G(D))
$$
for all $k \geq 0$, which proves the theorem.
\end{proof}

Conjecture~\ref{Conj:1} follows as corollary of the two Theorems~\ref{Th:Main} and \ref{Th:comp}: 
The degree sequence of the $n$-vertex Volkmann tree, which is of the form 
$(\Delta,\dots,\Delta,r,1,\dots,1)$ for some $1\leq r<\Delta$, majorizes all possible degree 
sequence of $n$-vertex trees with maximum degree $\Delta$.

More results can be obtained by similar arguments in the same way as Corollaries 5.1 -- 5.5 of 
\cite{zhang12} and Corollaries 29 -- 32 of \cite{Andriantiana2012} are obtained. Let us state some more of these corollaries, which also recover some results that can be found in \cite{zhang2011estrada,du2011estrada2}:
\begin{cor}
For any $n$-vertex tree $T$ and for any $k\geq 0$, $$\M_{k}(S_n)\geq \M_k(T),$$ where $S_n$ 
is the star with $n$ vertices, whose degree sequence is $(n-1,1,\dots,1)$.
\end{cor}
\begin{cor}
Among trees $T$ of order $n$ with $s$ leaves, $\M_k(T)$ 
is maximized by the greedy tree $G(s,2,2,\ldots,2,1,1,\ldots,1)$ (the number of $2$s is $n-s-1$, 
the number of $1$s is $s$) for any $k \geq 0$.
\end{cor}

\begin{cor}
Among trees $T$ of order $n$ with independence number $\alpha \geq n/2$ and among all trees $T$ with matching number 
$n-\alpha\leq n/2$, $\M_k(T)$ is maximized by the greedy tree $G(\alpha,2,2,\ldots,2,1,1,\ldots,1)$ 
(the number of $2$s is $n-\alpha-1$, the number of $1$s is $\alpha$) for any $k \geq 0$.
\end{cor}
$\M_k$ in each of the above corollaries can of course be replaced by $\EE$ or more generally $\E_f$ for any $f$ with nonnegative coefficients in \eqref{Eq:f}. If infinitely many even-indexed coefficients are strictly positive (e.g., for $\EE$), then we even have strict inequality. Moreover, corollaries analogous to Corollary~\ref{cor:specrad} and Corollary~\ref{cor:charpoly} for the spectral radius and the values of the characteristic polynomial also follow easily.

\bibliographystyle{abbrv} 
\bibliography{Andriantiana_Wagner_Spectral_Moments}
\end{document}